\documentclass[12pt]{amsart}
\usepackage{graphicx} 
\usepackage{mathrsfs}
\usepackage{graphicx}

\usepackage{thmtools}
\usepackage{amsfonts}
\usepackage{quiver}
\usepackage{bbding}
\usepackage{comment}
\usepackage{mathtools}
\usepackage{xcolor}

\usepackage{etoolbox}
\patchcmd{\thebibliography}{\chapter*}{\section*}{}{}
\usepackage{ltablex}
\usepackage{enumerate}
\usepackage[nospace,noadjust]{cite}
\usepackage{calc}
\usepackage{geometry}
\usepackage{cite}
\usepackage{tikz-cd}
\usepackage{tikz}
\usetikzlibrary{matrix}

\usepackage{indentfirst}
\usepackage{longtable}
\usepackage[hypertexnames = false]{hyperref}
\hypersetup{
	colorlinks=true,
	linkcolor=blue,
	filecolor=blue,      
	urlcolor=cyan,
}

\allowdisplaybreaks

\numberwithin{equation}{section} 

\newtheorem*{thmA}{Theorem A}
\newtheorem*{thmB}{Theorem B}
\newtheorem*{thmC}{Theorem C}

\newtheorem*{question*}{Question}

\newtheorem{definition}{Definition}[section]

\newtheorem{theorem}[definition]{Theorem}

\newtheorem{proposition}[definition]{Proposition}
\newtheorem{corollary}[definition]{Corollary}
\newtheorem{lemma}[definition]{Lemma}

\newtheorem{remark}[definition]{Remark}

\newcommand{\Gal}{\mathrm{Gal}}

\newcommand{\A} {\mathbf{A}}

\newcommand{\bZ} {\mathbb{Z}}

\newcommand{\bQ}{\mathbb{Q}}

\newcommand{\Gm}{\mathbb{G}_m}

\def\PP{\mathbb{P}}

\newcommand{\ZZ}{\mathbb{Z}}
\newcommand{\RR}{\mathbb{R}}

\DeclareMathOperator{\Res}{Res}

\DeclareMathOperator{\Hil}{Hil}

\DeclareMathOperator{\Spec}{Spec}
\DeclareMathOperator{\Br}{Br}
\DeclareMathOperator{\Sym}{Sym}
\DeclareMathOperator{\CH}{CH}

\DeclareMathOperator{\Pic}{\textup{Pic}}
\DeclareMathOperator{\NS}{NS}
\DeclareMathOperator{\homology}{H}

\DeclareMathOperator{\pr}{pr}

\newcommand{\Brnr}{\Br_{\mathrm{nr}}}

\begin{document}

	\title[Brauer groups of symmetric products]{Unramified Brauer groups of symmetric products and the Brauer-Manin obstructions}
	\author{Yongqi Liang}
	\author{Xingyu Liu}
	\author{Hui Zhang}
	
	\address{Yongqi Liang
		\newline University of Science and Technology of China,
		\newline School of Mathematical Sciences,
		\newline 96 Jinzhai Road,
		\newline  230026 Hefei, Anhui, China
	}
	\email{yqliang@ustc.edu.cn}
	
	\address{Xingyu Liu
		\newline University of Science and Technology of China,
		\newline School of Mathematical Sciences,
		\newline 96 Jinzhai Road,
		\newline  230026 Hefei, Anhui, China
	}
	\email{tsuki@mail.ustc.edu.cn}
	
	\address{Hui Zhang
		\newline University of Science and Technology of China,
		\newline School of Mathematical Sciences,
		\newline 96 Jinzhai Road,
		\newline  230026 Hefei, Anhui, China
	}
	\email{zhero@mail.ustc.edu.cn}

	\maketitle
	\begin{abstract}
		\noindent This article focuses on smooth, projective, and geometrically integral varieties $X$ defined over a number field $k$ with torsion-free geometric Picard groups. We establish an isomorphism between the Brauer groups of $X$ and its symmetric products. As applications,  we deduce the relationship between the Brauer--Manin obstruction to the Hasse principle and to weak approximation for $0$-cycles of degree $n$ on $X$ and the corresponding obstruction for rational points on smooth projective models of its $n$-fold symmetric product.
	\end{abstract}
	\section{Introduction}
	
	The Brauer group plays a very important role in the study of local-global properties in arithmetic geometry. A key tool is the Brauer–Manin pairing introduced by Manin \cite{Man71}, which provides a systematic obstruction to the existence of rational points and to weak approximation. For a projective variety $X$ defined over a number field $k$, the pairing is defined as
	\[
	\langle \, \cdot\, , \cdot \, \rangle_k : X(\A_{k}) \times \Br(X) \longrightarrow \mathbb{Q}/\mathbb{Z},
	\]
	\[
	\langle (x_v)_v, \, b \rangle_{k} \longmapsto \sum_{v \in \Omega_k} \operatorname{inv}_v\bigl(b(x_v)\bigr).
	\]
	where $\operatorname{inv}_{v}: \Br(k_{v})\rightarrow \bQ/\bZ$ is the local invariant at the place $v$ and $b(x_{v})$ denotes the evaluation of $b$ at the local point $x_{v}$.
	
	The Brauer–Manin set of $X$ is defined as the left kernel of this pairing
	\[
	X(\A_{k})^{\Br} := \biggl\{ (x_v)_v \in X(\A_{k}) \; \big| \; \langle (x_v)_v, b \rangle_k = 0 \text{ for all } b \in \Br(X) \biggr\}.
	\]
	By class field theory, the set $X(k)$ of rational points is contained in $X(\A_{k})^{\Br}$.
	So the condition $X(\A_{k})^{\Br}= \varnothing$ gives an obstruction to the existence of global rational points. We say that  
	\emph{the Brauer–Manin obstruction is the only obstruction to the Hasse principle for rational points} (abbreviated \emph{HPBMrp}) if the implication
	\[
	X(\A_{k})^{\Br} \neq \varnothing \quad \Longrightarrow \quad X(k) \neq \varnothing
	\]
	holds. 
	Similarly, because the Brauer–Manin set is closed, the closure of $X(k)$ satisfies $\overline{X(k)} \subseteq X(\A_{k})^{\Br}$. We say that  
	\emph{the Brauer–Manin obstruction is the only obstruction to weak approximation for rational points} (abbreviated \emph{WABMrp}) if
	\[
	\overline{X(k)} = X(\A_{k})^{\Br}.
	\]
	
	Similarly, for 0-cycles of degree $n$, we can define:
	\begin{enumerate}
		\item\emph{ The Brauer--Manin obstruction is the only obstruction to the Hasse principle for 0-cycles of degree $n$ (abbreviated $\text{HPBM0cyc}^{n}$)}, see Definition~\ref{Definition:HPBM0cyc}.
		\item  \emph{The Brauer--Manin obstruction is the only obstruction to the weak approximation for 0-cycles of degree $n$ (abbreviated $\text{WABM0cyc}^{n}$)}, see Definition~\ref{Definition:WABM0cyc}.
	\end{enumerate}

The first author showed in \cite{Lia13} that for  rationally connected varieties $X$ defined over a number field $k$, if for every finite field extension $L/k$, the HPBMrp (resp.\ WABMrp) holds for $X_{L}$, then $X$ satisfies $\text{HPBM0cyc}^{1}$(resp.\ $\text{WABM0cyc}^{1}$). Such conclusions have also been established with similar method for K3 surfaces by Evis Ieronymou in \cite{Ier21}, and for Kummer varieties by Francesca Balestrieri and Rachel Newton in \cite{BN21}.  

In a different direction, Sheng Chen and Ziyang Zhang  proved in \cite[Theorem 1.3]{CZ24} that the WABMrp  for a rationally connected variety $X$ after suitable field extensions implies the WABMrp for smooth projective models of $\operatorname{Sym}^{n}_{X}$ for all $n$. It is a bit surprising that their argument by fibration method does not give any information about the Brauer group of the symmetric product.

As rational points on the $n$-fold symmetric product $\operatorname{Sym}^{n}_{X}$ of $X$ corresponds bijectively to effective 0‑cycles of degree $n$ on $X$, it is therefore natural to ask:
	\begin{question*}
		Is there any relation between the Brauer--Manin obstruction for rational points on $\Sym_{X}^{n}$ and the Brauer--Manin obstruction for the 0-cycles on $X$?
	\end{question*}
	
In order to answer to the Question, we study the unramified Brauer group of the symmetric product $\Sym_{X}^{n}$ and obtain our main result.
	\begin{thmA}[Theorem~\ref{theorem}]
		Let \(X\) be a smooth, projective, geometrically integral variety defined over a number field $k$ such that 
		\(\Pic(X_{\overline{k}})\) is torsion‑free. Then 
		there is an isomorphism of abelian groups
		\[
		\lambda \colon \Br(X)/\Br(k) \xrightarrow{\ \simeq\ } 
		\Brnr\!\bigl(\Sym_{X}^{n}\bigr)/\Br(k).
		\]
	\end{thmA}
As an arithmetic application, partially answering to the Question, we relate the Brauer--Manin obstruction for 0-cycles on $X$ with that for rational points on symmetric products.

	\begin{thmB}[Theorem~\ref{thm:main}]
	Let  \(X\) be a smooth, projective, geometrically integral variety defined over a number field \(k\) such that \(\Pic(X_{\overline{k}})\) is torsion-free. 
	\begin{enumerate}[(a)]
		\item The canonical identification
		\[
		\mathrm{Z}_{0,\A}^{n,\mathrm{sep},\mathrm{eff}}(X)=\prod_{v\in\Omega_k}\Sym_{X}^{n,o}(k_v)
		\]
		induces an identification
		\[
		\mathrm{Z}_{0,\A}^{n,\mathrm{sep},\mathrm{eff}}(X)^{\Br}= 
		\Bigl(\prod_{v\in\Omega_k}\Sym_{X}^{n,o}(k_v)\Bigr)^{\Brnr}.
		\]
		In particular, 
\begin{itemize}
\item[-] the variety \(\Sym_{X}^{n,o}\) satisfies HPBMrp with respect to the unramified Brauer group if and only if \(X\) satisfies $\text{HPBM0cyc}^{n,\mathrm{eff,sep}}$; 
\item[-] if \(\Sym_{X}^{n,o}\) satisfies the WABMrp, then \(X\) satisfies the $\text{WABM0cyc}^{n,\mathrm{eff,sep}}$.
\end{itemize}[As indicated by the notation, separable effective 0-cycles are considered in this statement, cf.  Section \ref{Section:Notation} for full details.]
		\item Assume moreover that \(\Br(X)/\Br(k)\) is finite. If a smooth projective model \(\Sym_{X}^{n,\mathrm{sm}}\) of the symmetric product satisfies HPBMrp (resp. WABMrp) with respect to the unramified Brauer group for $n$ sufficiently large, then \(X\) satisfies $\text{HPBM0cyc}^{1}$(resp. $\text{WABM0cyc}^{1}$).
	\end{enumerate}
	\end{thmB}

This particular family of varieties under consideration in both theorems includes rationally connected varieties, K3 surfaces and Kummer varieties, and any finite product with factors chosen among them. As another arithmetic application of Theorem A, our Theorem \ref{thm:generalization} also recovers the result \cite[Theorem 1.3]{CZ24} for rationally connected varieties, and indeed extends their result to our particular family of varieties. 

On the way to the proofs of the above theorems, we also obtain the following byproduct which is certainly of independent interest. It partially recovers a recent result of Sheng Chen and Kai Huang \cite{CHto}.

\begin{thmC}[Proposition~\ref{Proposition:WeirestricionBrauer}(c), Corollary~\ref{Corollary:identityofWeil}]
Let $l/k$ be an extension of number fields.
Let  \(X\) be a smooth, projective, geometrically integral variety defined over \(l\) with torsion-free geometric Picard group. Denote by $X'=\textup{Res}_{l/k}(X)$ its Weil restriction. 
Then the natural homomorphisms 
$$\Br_1(X)/\Br(l)\buildrel{\simeq}\over\to\Br_1(X')/\Br(k)$$
$$\Br(X)/\Br(l)\buildrel{\simeq}\over\to\Br(X')/\Br(k)$$ are isomorphisms.
As a consequence,  we have $X(\A_{l})^{\Br_1}=X'(\A_{k})^{\Br_1}$ and $X(\A_{l})^{\Br}=X'(\A_{k})^{\Br}$.
\end{thmC}
The question of  comparison of various obstruction subsets of $X$ and $X'$ raised by Yang Cao and the first author in \cite[paragraph before Theorem 1.1]{CL22} originates further back to an earlier paper of Jean-Louis Colliot-Th\'el\`ene and Bjorn Poonen \cite[Remark 5, page 95]{CTPoonen}. Many authors have contributed to this question, cf. \cite[Proposition 5.15]{Stoll07}, \cite{CL22}, and \cite{CHto}.
However, these existing results are all completed by descent arguments. To the best of the knowledge of the authors, the isomorphism between the Brauer groups of $X$ and $X'$ was not established in the literature.

\subsection*{Organization of the paper}\ \\

We firstly recall some notation and preliminaries in Section~\ref{Section:Notation}. 

In Section~\ref{Section:varieties}, we then turn to the study of varieties with torsion-free geometric Picard group and give the proof of Theorem C on Weil restrictions.

In Section~\ref{Section:Symmetricproduct}, we study symmetric products and define a homomorphism from the Brauer group of $X$ to the unramified Brauer group of its $n$-fold symmetric product. Moreover, we show that, in a relative version, the generic fiber of the induced morphism between symmetric products can be interpreted as a certain Weil restriction of the generic fiber of the structure morphism. Therefore, the result of Section~\ref{Section:varieties} can be applied to a trivial fibration to prove Theorem A. 

Finally, we discuss arithmetic applications to Brauer--Manin obstructions in Section~\ref{Section:Brauermanin}.

	\section{Notation and Preliminaries}\label{Section:Notation}
	In this paper, a field $k$ is always of characteristic 0 if not explicitly mentioned. We fix an algebraic closure $\overline{k}$ of $k$. The absolute Galois group $\textup{Gal}(\overline{k}/k)$ is denoted by $\Gamma_k$.
	
	For a quasi-projective variety $X$ over $k$, a \emph{smooth projective model} of $X$ is a smooth projective variety $Y$ over $k$ that contains $X$ as a dense open subset. The existence of such a model is guaranteed by Nagata's compactification theorem followed by Hironaka's resolution of singularities.

	For the definition of the Brauer–Manin pairing for 0-cycles, see \cite[Chapter 6]{CTS21}. Let $\mathrm{Z}_{0}^{n}(X)$ (respectively $\mathrm{Z}_{0,\A_{k}}^{n}(X)$ when $k$ is a number field) denote the set of 0-cycles of degree $n$ (respectively adelic 0-cycles of degree $n$) on $X$. 
	
	\begin{definition}\label{Definition:HPBM0cyc}
		We say that the \emph{Brauer–Manin obstruction is the only obstruction to the Hasse principle for 0-cycles of degree $n$} (abbreviated \emph{$\text{HPBM0cyc}^n$}), if 
		\[
		\mathrm{Z}_{0,\A_{k}}^{n}(X)^{\Br} \neq \varnothing
		\quad \Longrightarrow \quad
		\mathrm{Z}_{0}^{n}(X) \neq \varnothing.
		\]
	\end{definition}
	Usually, one says that \emph{$X$ satisfies HPBM for 0-cycles} if it satisfies $\text{HPBM0cyc}^{1}$. For weak approximation, we follow the definition given in \cite{Lia13}. 
	\begin{definition}\label{Definition:WABM0cyc}
		We say that \emph{the Brauer–Manin obstruction is the only obstruction to weak approximation for 0‑cycles of degree $n$} (abbreviated \emph{$\text{WABM0cyc}^{n}$}) if the following holds: for every positive integer $\delta$ and every finite subset $S\subset\Omega_k$, given any family $\{z_v\}$ of local 0‑cycles of degree $n$ that is orthogonal to $\operatorname{Br}(X)$ (i.e. $\{z_v\}\perp \operatorname{Br}(X)$), there exists a global 0‑cycle $z_{\delta,S}$ of degree $n$ such that for every $v\in S$ the images of $z_{\delta,S}$ and $z_v$ in $\CH_0(X_{k_v})/\delta$ coincide.  
	\end{definition}

	Recall that  a 0-cycle $z=\sum_{P}n_{P}P$ is effective if $n_{P}\geq0$ for all $P$. It is  called \emph{effective separable} if each coefficient $n_P$ is either 0 or 1. For such 0-cycles we respectively define \emph{$\text{HPBM0cyc}^{n,\mathrm{eff,sep}}$} and \emph{$\text{WABM0cyc}^{n,\mathrm{eff,sep}}$} in a similar way.

	
	A separable extension $K/k$ of fields is called \emph{regular} if $k$ is algebraically closed in $K$. When $K/k$ is regular, for any algebraic extension $L/k$, the tensor product $K \otimes_k L$ remains a field. We denote it by $K.L$ when the base field $k$ is clear. For further details on regular extensions, see \cite[Chap.~8, Sec.~4]{Lan02}.

	\begin{lemma}\label{Lemma:fieldextension}
		For a regular extension $K/k$, we have an exact sequence of Galois groups
		\begin{equation*}
			0\rightarrow \Gamma_{K.\overline{k}}\rightarrow \Gamma_{K} \rightarrow \Gamma_{k}\rightarrow 0
		\end{equation*}
	\end{lemma}
	\begin{proof}[Sketch of proof]
		Consider the tower of field extensions  
		\[
		K \subset K.\overline{k} \subset \overline{K}.
		\]  
		Since \(K/k\) is regular, the extensions \(K\) and \(\overline{k}\) are linearly disjoint over \(k\). Consequently,  
		\[
		\operatorname{Gal}(K.\overline{k}/K) \simeq \Gamma_k,
		\]  
		and the desired conclusion follows.
	\end{proof}
	Recall that a variety $X$ over $k$ is called \emph{split} if it contains an open geometrically integral $k$-subscheme. 
	
	Assume that $X$ is a proper and geometrically integral variety over a field $k$. The Leray spectral sequence
	\[
	H^{p}(k, H^{q}(X_{\overline{k}},\mathbb{G}_{m})) \Rightarrow H^{p+q}(X,\mathbb{G}_{m})
	\]
	yields  long exact sequences
	\begin{equation}\label{sequence2.1}
	\begin{aligned}
		0  \rightarrow \operatorname{Pic}(X) \rightarrow \operatorname{Pic}(X_{\overline{k}})^{\Gamma_{k}} &\rightarrow \operatorname{Br}(k) \rightarrow \operatorname{Br}_{1}(X) \\
		& \rightarrow \homology^{1}(k, \operatorname{Pic}(X_{\overline{k}})) \rightarrow \ker\!\big(\homology^{3}(k,\mathbb{G}_{m}) \rightarrow \homology^{3}(X,\mathbb{G}_{m})\big)
	\end{aligned}
	\end{equation}
	and 
	\begin{equation}\label{sequence2.2}
		0\rightarrow \Br_{1}(X)/\Br(k)\rightarrow \Br(X)/\Br(k) \rightarrow \Br(X_{\overline{k}})^{\Gamma_{k}} \rightarrow \homology^{2}(k,\Pic(X_{\overline{k}})).
	\end{equation}
	
	The sequence \eqref{sequence2.1} naturally induces a monomorphism \begin{equation}\label{epsilon}\epsilon_{X}:\Br_{1}(X)/\Br(k)\rightarrow \homology^{1}(k,\Pic(X_{\overline{k}})).\end{equation} 
	
	It is well-known that the homomorphism $\epsilon_X$  is surjective when $k$ is a number field, or when $X$ has index $1$ (i.e. admits a 0-cycle of degree 1).

	\section{Varieties with torsion-free geometric Picard groups}\label{Section:varieties}
	\subsection{Preliminaries}
	We now focus on smooth, geometrically integral varieties $X$ whose geometric Picard groups $\Pic(X_{\overline{k}})$ are torsion-free. 
	\begin{remark}\label{Remark:torsionfree}
		\begin{enumerate}[(a)]\label{Remark:Picardtorsionfree}
			\item For a smooth, geometrically integral variety $X$ over a field $k$, the condition that $\Pic(X_{\overline{k}})$ is torsion-free is equivalent to each of the following:
			\begin{itemize}
				\item $\homology^{1}(X_{\overline{k}},\bQ/\bZ)=0$;
				\item $\homology^{1}(X,\mathcal{O}_{X})=0$ and the Néron–Severi group $\NS(X_{\overline{k}})$ is torsion-free.
			\end{itemize}
			To see this, consider the exact sequence
			\[
			0\longrightarrow \textbf{\textup{Pic}}_{X/k}^{0}(\overline{k}) \longrightarrow \Pic(X_{\overline{k}}) \longrightarrow \NS(X_{\overline{k}}) \longrightarrow 0,
			\]
			where $\textbf{\textup{Pic}}_{X/k}^{0}(\overline{k})$ is the set of $\overline{k}$-rational points of the Picard variety (an abelian variety). If $\Pic(X_{\overline{k}})$ is torsion-free, then its subgroup $\textbf{\textup{Pic}}_{X/k}^{0}(\overline{k})$ must also be torsion-free. However, for every integer $n>0$, the $n$-torsion subgroup $\textbf{\textup{Pic}}_{X/k}^{0}(\overline{k})$ is isomorphic to $(\bZ/n\bZ)^{\oplus 2g}$ (see \cite[\href{https://stacks.math.columbia.edu/tag/0BF9}{Tag 0BF9}]{stacks-project}). Hence $g=\dim(\textbf{\textup{Pic}}_{X/k}^{0})=0$, so $\textbf{\textup{Pic}}_{X/k}^{0}(\overline{k})$ is trivial. Consequently, $\homology^{1}(X,\mathcal{O}_{X})=0$ since $\dim \homology^{1}(X,\mathcal{O}_{X}) = \dim \textbf{\textup{Pic}}_{X/k}^{0}$, and the sequence yields an isomorphism $\Pic(X_{\overline{k}}) \simeq \NS(X_{\overline{k}})$, which is a finitely generated torsion-free abelian group.
			
			\item The class of varieties with torsion-free geometric Picard group includes many interesting examples:
			\begin{itemize}
				\item rationally connected varieties;
				\item K3 surfaces (see \cite[Chapter 1, Proposition 2.4]{Huy16});
				\item Kummer varieties (see \cite[Proposition 2.3]{SZ17}).
			\end{itemize}
		\end{enumerate}
	\end{remark}
	\begin{lemma}\label{Lemma:torsionfreepicard}
		Let $X$ be a smooth, projective, geometrically integral variety over a number field $k$. 
		For any field extension $K/k$, the Picard group $\Pic(X_{\overline{k}})$ is torsion-free if and only if $\Pic(X_{\overline{K}})\) is torsion-free.
	\end{lemma}
	
	\begin{proof}
		By flat base change, we have an isomorphism
		\[
		\homology^{1}(X, \mathcal{O}_{X}) \otimes_{k} K \simeq \homology^{1}(X_{K}, \mathcal{O}_{X_{K}}).
		\]
		If $\Pic(X_{\overline{k}})$ is torsion-free, then by Remark \ref{Remark:torsionfree}(a) we have $\textup{H}^{1}(X, \mathcal{O}_{X}) = 0$, and consequently $\homology^{1}(X_{K}, \mathcal{O}_{X_{K}}) = 0$. 
		Applying Remark \ref{Remark:torsionfree}(a) again, we obtain
		\[
		\Pic(X_{\overline{k}}) \simeq \NS(X_{\overline{k}})
		\quad \text{and} \quad
		\Pic(X_{\overline{K}}) \simeq \NS(X_{\overline{K}}),
		\]
		where $\overline{K}$ is a fixed algebraic closure of $K$.
		
		For any algebraically closed field $L$ containing $k$, the Néron--Severi group $\NS(X_{L})$ is isomorphic to the group of connected components of the Picard scheme of $X$ (see \cite[Corollary 4.18.3, Proposition 5.3, Proposition 5.10]{KL05}). 
		In particular, $\NS(X_{\overline{k}})$ and $\NS(X_{\overline{K}})$ are isomorphic. 
		Therefore, $\Pic(X_{\overline{k}})$ is torsion-free if and only if $\Pic(X_{\overline{K}})$ is torsion-free.
	\end{proof}
	
	\begin{lemma}\label{Lemma:Brauerregular}
		Let $X$ be a smooth variety over an algebraically closed field $k$ and $K/k$ be an  extension of fields
		\begin{enumerate}[(a)]
			\item  The restriction homomorphism $\Br(X)\rightarrow \Br(X_{K})$ is injective.
			\item  If $K$ is moreover algebraically closed, then the homomorphism above is an isomorphism
		\end{enumerate}
	\end{lemma}
	\begin{proof}
		See \cite[Proposition 5.2.2]{CTS21} and \cite[Proposition 5.2.3]{CTS21}.
	\end{proof}
	
	\begin{lemma}\label{Lemma:Brauerregularextension}
		Let $X$ be a smooth projective geometrically integral variety over a field $k$ such that $\Pic(X_{\overline{k}})$ is a torsion-free abelian group, let $K/k$ be a regular field extension.
		\begin{enumerate}[(a)]
			\item The induced homomorphism
			\[
			\homology^{i}\bigl(k,\Pic(X_{\overline{k}})\bigr) \longrightarrow \homology^{i}\bigl(K,\Pic(X_{\overline{K}})\bigr)
			\]
			is bijective for $i=1$ and injective for $i=2$. Consequently, the homomorphism
			\[
			\Br_{1}(X)/\Br(k) \longrightarrow \Br_{1}(X_{K})/\Br(K)
			\]
			is injective.
			\item The isomorphism $\Br(X_{\overline{k}}) \simeq \Br(X_{\overline{K}})$ induces an isomorphism
			\[
			\Br(X_{\overline{k}})^{\Gamma_{k}} \simeq \Br(X_{\overline{K}})^{\Gamma_{K}}.
			\]
			Moreover, the homomorphism
			\[
			\Br(X)/\Br(k) \longrightarrow \Br(X_{K})/\Br(K)
			\]
			is injective.
			\item If the homomorphism $\epsilon_{X}: \Br_{1}(X)/\Br(k) \to \homology^{1}\bigl(k,\Pic(X_{\overline{k}})\bigr)$ in \eqref{epsilon} is an isomorphism, then the homomorphism in (b) is also an isomorphism.
		\end{enumerate}
	\end{lemma}
	
	\begin{proof}
		\begin{enumerate}[(a)]
			    \item We know that the Galois group $\Gamma_{K}$ acts on $\Pic(X_{\overline{K}})\simeq \Pic(X_{\overline{k}}) $ factoring through its quotient $\Gamma_{k}$ according to Lemma \ref{Lemma:fieldextension}.  Consider the inflation-restriction exact sequence for profinite groups \cite[Proposition 1.6.6]{NSW08}
			\begin{equation*}
			\begin{aligned}
				0 \to \homology^{1}\bigl(k,\Pic(X_{\overline{k}})^{\Gamma_{K.\overline{k}}}\bigr) \to \homology^{1}\bigl(K,\Pic(X_{\overline{K}})\bigr) &\to \homology^{1}\bigl(\Gamma_{K.\overline{k}},\Pic(X_{\overline{K}})\bigr)^{\Gamma_{k}} \\
				& \to \homology^{2}\bigl(k,\Pic(X_{\overline{k}})^{\Gamma_{K.\overline{k}}}\bigr) \to \homology^{2}\bigl(K,\Pic(X_{\overline{K}})\bigr).
			\end{aligned}			
		\end{equation*}
			Since $\Gamma_{K.\overline{k}}$ acts trivially on the finitely generated and torsion-free module $\Pic(X_{\overline{K }})$, we have
			\[
			\homology^{1}\bigl(\Gamma_{K.\overline{k}},\Pic(X_{\overline{K}})\bigr)^{\Gamma_{k}} = 0.
			\]
			Hence we get the desired isomorphism for $i=1$ and the desired  monomorphism for $i=2$. The claim about $\Br_{1}(X)/\Br(k)$ follows from the fact that it is a subgroup of $\homology^{1}\bigl(k,\Pic(X_{\overline{k}})\bigr)$ (and similarly over $K$).
			
			 \item By Lemma~\ref{Lemma:Brauerregular}, we know that $\Br(X_{\overline{k}})\simeq \Br(X_{\overline{K}})$ immediately, the isomorphism $\Br(X_{\overline{k}})^{\Gamma_{k}} \simeq \Br(X_{\overline{K}})^{\Gamma_{K}}$ then follows from the fact that $\Gamma_{K.\overline{k}}$ acts trivially on $X_{\overline{k}}$. Consider the commutative diagram with exact rows
			\begin{equation*}
				\hspace{-0.3in}
				\begin{tikzcd}
					0 \arrow[r] & \Br_{1}(X)/\Br(k) \arrow[r] \arrow[d, hook] & \Br(X)/\Br(k) \arrow[r] \arrow[d] & \Br(X_{\overline{k}})^{\Gamma_{k}} \arrow[r] \arrow[d, "\simeq"] & \homology^{2}\bigl(k,\Pic(X_{\overline{k}})\bigr) \arrow[d, hook] \\
					0 \arrow[r] & \Br_{1}(X_{K})/\Br(K) \arrow[r] & \Br(X_{K})/\Br(K) \arrow[r] & \Br(X_{\overline{K}})^{\Gamma_{K}} \arrow[r] & \homology^{2}\bigl(K,\Pic(X_{\overline{K}})\bigr)
				\end{tikzcd}
			\end{equation*}
			The left and right vertical homomorphism is injective by (a), then the middle vertical homomorphism is injective by diagram chasing.
			
			 \item If $\epsilon_{X}$ is an isomorphism, then by (a) the homomorphism $\epsilon_{X_{K}}$ is also an isomorphism. Thus
			\[
			\Br_{1}(X)/\Br(k) \simeq \homology^{1}\bigl(k,\Pic(X_{\overline{k}})\bigr) \simeq \homology^{1}\bigl(K,\Pic(X_{\overline{K}})\bigr) \simeq \Br_{1}(X_{K})/\Br(K).
			\]
			The five-lemma applied to the diagram in (b) then shows that the middle vertical homomorphism is an isomorphism.
		\end{enumerate}
	\end{proof}
	
	\begin{proposition}\label{Proposition:productofcohomology}
		Let $X$ and $Y$ be smooth, projective, geometrically integral varieties over a field $k$ such that $\Pic(X_{\overline{k}})$ is torsion-free. Set $Z = X \times_k Y$, and let $\pr_X$ and $\pr_Y$ be the projections.
		\begin{enumerate}[(a)]
			\item The canonical homomorphism
			\[
			\pr_Y^* : \homology^1(Y, \mathcal{O}_Y) \longrightarrow \homology^1(Z, \mathcal{O}_Z)
			\]
			is an isomorphism.
			
			\item The following canonical homomorphisms induced by the projections are isomorphisms of $\Gamma_k$-modules:
			\[
			\begin{aligned}
				\pr_X^* + \pr_Y^* &: \NS(X_{\overline{k}}) \oplus \NS(Y_{\overline{k}}) \xrightarrow{\simeq} \NS(Z_{\overline{k}}), \\
				\pr_X^* + \pr_Y^* &: \Pic(X_{\overline{k}}) \oplus \Pic(Y_{\overline{k}}) \xrightarrow{\simeq} \Pic(Z_{\overline{k}}), \\
				\pr_X^* + \pr_Y^* &: \Br(X_{\overline{k}}) \oplus \Br(Y_{\overline{k}}) \xrightarrow{\simeq} \Br(Z_{\overline{k}}).
			\end{aligned}
			\]
			
			\item The canonical homomorphism
			\[
			\bigl(\Br(X)/\Br(k)\bigr) \oplus \bigl(\Br(Y)/\Br(k)\bigr) \xrightarrow{\simeq} \Br(Z)/\Br(k)
			\]
			is an isomorphism if both $\epsilon_X$ and $\epsilon_Y$ defined in \eqref{epsilon} are isomorphisms. This holds in particular when $k$ is a number field.
		\end{enumerate}
		
	\end{proposition}
	\begin{proof}
		As in Remark \ref{Remark:Picardtorsionfree}(a), the condition that $\Pic(X_{\overline{k}})$ is torsion-free implies $\homology^1(X, \mathcal{O}_X) = 0$ and the Picard variety of $X$ is trivial. Thus assertion (a) follows directly from the Künneth formula \cite[\href{https://stacks.math.columbia.edu/tag/0BEC}{Tag 0BEC}]{stacks-project}. Assertion (b) is a consequence of \cite[Proposition 5.7.3]{CTS21} and \cite[Corollary 5.7.10]{CTS21}.
		
		It remains to prove (c). From (b) we obtain isomorphisms
		\[
		\begin{aligned}
			\pr_X^* + \pr_Y^* &: \Br(X_{\overline{k}})^{\Gamma_k} \oplus \Br(Y_{\overline{k}})^{\Gamma_k} \xrightarrow{\simeq} \Br(Z_{\overline{k}})^{\Gamma_k}, \\
			\pr_X^* + \pr_Y^* &: \homology^i\bigl(k,\Pic(X_{\overline{k}})\bigr) \oplus \homology^i\bigl(k,\Pic(Y_{\overline{k}})\bigr) \xrightarrow{\simeq} \homology^i\bigl(k,\Pic(Z_{\overline{k}})\bigr) \quad (i \ge 1).
		\end{aligned}
		\]
		From the exact sequence \eqref{sequence2.2}, the conclusion holds if and only if
		\[
		\pr_X^* + \pr_Y^* : \bigl(\Br_1(X)/\Br(k)\bigr) \oplus \bigl(\Br_1(Y)/\Br(k)\bigr) \longrightarrow \Br_1(Z)/\Br(k)
		\]
		is an isomorphism. Under the assumption that $\epsilon_X$ and $\epsilon_Y$ are isomorphisms, with the help of the isomorphism on $\homology^1\bigl(k,\Pic(-_{\overline{k}})\bigr)$ induced by (b), one can then check that $\epsilon_{Z}$ is also an isomorphism. Hence the claim follows.
	\end{proof}
	
	Indeed, it follows directly from this proposition that if \(\Pic(X_{\overline{k}})\) and \(\Pic(Y_{\overline{k}})\) are both torsion-free abelian groups, where \(X\) and \(Y\) are smooth, projective, geometrically integral varieties over \(k\), then the product \(Z := X \times_k Y\) also has a torsion-free geometric Picard group \(\Pic(Z_{\overline{k}})\).

	\subsection{Brauer groups of Weil restrictions}

Let $l/k$ be a field extension of degree $d$. We fix an algebraic closure $\overline{k}$ of $k$ containing $l$. Let $X$ be a quasi-projective variety over $l$.  According to \cite[Theorem 4, Section 7.6]{BLR90}, the Weil restriction of scalars $X' \coloneqq \Res_{l/k}(X)$ is a $k$-variety. In this section, we show that under the additional assumption that $\Pic(X_{\overline{k}})$ is torsion-free and that $X$ is projective, the Brauer groups of $X$ and $X'$ are isomorphic modulo constant parts.

We reload the basic setting in \cite[Section 4]{CL22}, while we consider the sheaf $\mathbb{G}_m$  instead of the torsion sheaf $\mu_n$ concerned in \cite{CL22}. The arguments remain the same.

Let $\alpha: X' \times_k l \rightarrow X$ be the natural $l$-morphism induced by the identity morphism of $X'$ over $k$ and the universal property of the Weil restriction. This yields the following composition of homomorphisms
	\begin{equation}\label{Definition:phi}
		\phi : \homology^{p}(X,\Gm)\xrightarrow{\alpha^{*}} \homology^{p}(X'_{l},\Gm)\xrightarrow{\mathrm{cor}_{l/k}} \homology^{p}(X',\Gm),
	\end{equation}
	where $X'_l = X' \times_k l$.
	
	We consider the Leray spectral sequence for the étale cohomology of the multiplicative group
	\begin{equation*}
		E^{p,q}_{2} = \homology^{p}\big(l, H^{q}(X_{\overline{k}},\Gm)\big) \Rightarrow \homology^{p+q}(X,\Gm).
	\end{equation*}

	Recall that we have $d$ distinct $k$-embeddings $\sigma_s: l \hookrightarrow \overline{k}$, each of which extends naturally to a k-isomorphism $\sigma_s: \overline{k} \to \overline{k}$. Let $\Gamma_k / \Gamma_l = \{\sigma_s \Gamma_l \mid 1 \leq s \leq d\}$.  For simplicity, we denote the coset by $\gamma_s$. For each $\sigma_s$, let $X_{\overline{k}, \gamma_s}$ be the base change of $X$ from $l$ to $\overline{k}$ via the embedding $\sigma_s$. Since this construction depends only on the coset $\gamma_s$, we also write $X_{\overline{k}, s}$ for short. The group $\Gamma_k$ acts naturally on the set $\{\gamma_s\}$. Because Weil restriction is compatible with base change, we have
	\[
	X'_{\overline{k}} = \prod_{s=1}^{d} X_{\overline{k}, s}.
	\]
	
	For each $\gamma \in \Gamma_k$, there is an action of $\gamma$ on $X'_{\overline{k}}$ that permutes the factors of $\prod_{s=1}^{d} X_{\overline{k}, s}$. There is a natural left action of $\Gamma_l$ on $\homology^{q}(X_{\overline{k}},\mathbb{G}_m)$. Furthermore, the construction above yields a left action of $\Gamma_k$ on $\bigoplus_{s=1}^{d} \homology^{q}(X_{\overline{k}, s},\mathbb{G}_m)$by permuting the summands. This implies that, as a Galois module, $\bigoplus_{s=1}^{d} \homology^{q}(X_{\overline{k}, s},\mathbb{G}_m)$ is precisely the induced representation $\operatorname{Ind}_{\Gamma_l}^{\Gamma_k} \homology^{q}(X_{\overline{k}},\mathbb{G}_m)$.
	
	Let $p_s \colon X'_{\overline{k}} \to X_{\overline{k}, s}$ denote the projection onto the $s$-th factor. These induce a homomorphism of $\Gamma_k$-modules  
	
	\[
	\bigoplus_{s=1}^{d} \homology^{q}(X_{\overline{k}, s},\mathbb{G}_m) \xrightarrow{\,p_1^* + \dots + p_d^*\,} \homology^{q}\!\Bigl(\prod_{s=1}^{d} X_{\overline{k}, s},\,\mathbb{G}_m\Bigr) = \homology^{q}(X'_{\overline{k}},\mathbb{G}_m).
	\]
	
	Combined with Shapiro’s lemma, this homomorphism yields the corresponding terms in the spectral sequence comparison. Consequently, we obtain a homomorphism
	
	\[
	\phi_{p,q} \colon \homology^{p}\bigl(l,\,\homology^{q}(X_{\overline{k}},\mathbb{G}_m)\bigr) \longrightarrow \homology^{p}\bigl(k,\,\homology^{q}(X'_{\overline{k}},\mathbb{G}_m)\bigr),
	\]
	which is compatible with the previously defined homomorphism $\phi$ in \eqref{Definition:phi}.

	\begin{proposition}\label{Proposition:WeirestricionBrauer}
		Let $X$ be a smooth, projective, and geometrically integral variety over $l$ such that $\Pic(X_{\overline{l}})$ is a torsion-free abelian group. Then:
		\begin{enumerate}[(a)]
			\item The group $\Pic(X'_{\overline{k}})$ is torsion-free.
			\item The canonical homomorphisms
			\[
			\phi_{p,1}: \homology^{p}\!\bigl(l,\Pic(X_{\overline{l}})\bigr) \longrightarrow \homology^{p}\!\bigl(k,\Pic(X'_{\overline{k}})\bigr)
			\]
			and
			\[
			\phi_{p,2}: \homology^{p}\!\bigl(l,\Br(X_{\overline{l}})\bigr) \longrightarrow \homology^{p}\!\bigl(k,\Br(X'_{\overline{k}})\bigr)
			\]
			are isomorphisms for every \(i \ge 0\).
			\item If the homomorphism \(\epsilon_X\) in \eqref{epsilon} is an isomorphism, then the homomorphisms
\[
			\phi: \Br_1(X)/\Br(l) \longrightarrow \Br_1(X')/\Br(k)
			\]
			\[
			\phi: \Br(X)/\Br(l) \longrightarrow \Br(X')/\Br(k)
			\]
			are isomorphisms.
		\end{enumerate}
		
	\end{proposition}
	
	\begin{proof}
		(a) This follows directly from Proposition~\ref{Proposition:productofcohomology}(b).
		
		(b) Also from the Proposition~\ref{Proposition:productofcohomology}(b), the homomorphism
		\[
		\operatorname{Ind}_{\Gamma_l}^{\Gamma_k} \homology^{q}(X_{\overline{k}},\mathbb{G}_m)=\bigoplus_{s=1}^{d} \homology^{q}(X_{\overline{k},s},\mathbb{G}_m) \xrightarrow{p_{1}^{*}+p_{2}^{*}+\cdots+p_{d}^{*}} \homology^{q}\Bigl(\prod_{s=1}^{d} X_{\overline{k},s},\mathbb{G}_m\Bigr) = \homology^{q}(X'_{\overline{k}},\mathbb{G}_m)
		\]
		is an isomorphism of $\Gamma_k$-modules for \(q = 1, 2\). The conclusion then follows from Shapiro's lemma.
		
		(c) Consider the commutative diagram
		\[
		\begin{tikzcd}
			0 & \Br_{1}(X)/\Br(l) & H^{1}(l,\Pic(X_{\overline{k}})) & 0 \\
			0 & \Br_{1}(X')/\Br(k) & H^{1}(k, \Pic(X'_{\overline{k}})) & \operatorname{coker}(\epsilon_{X'})
			\arrow[from=1-1, to=1-2]
			\arrow[from=1-2, to=1-3]
			\arrow["\phi", from=1-2, to=2-2]
			\arrow[from=1-3, to=1-4]
			\arrow["\phi", from=1-3, to=2-3]
			\arrow["\phi", from=1-4, to=2-4]
			\arrow[from=2-1, to=2-2]
			\arrow[from=2-2, to=2-3]
			\arrow[from=2-3, to=2-4]
		\end{tikzcd}
		\]
		We have the isomorphism \(\Br_{1}(X)/\Br(l)\simeq H^{1}(l,\Pic(X_{\overline{k}}))\). Since the second vertical homomorphism \(\phi\) is an isomorphism by assertion (b), the commutativity of the diagram implies that the homomorphism \(\Br_{1}(X')/\Br(k)\rightarrow \homology^{1}(k,\Pic(X'_{\overline{k}}))\) is surjective. We deduce an isomorphism
		\[
		\Br_{1}(X')/\Br(k)\simeq \Br_{1}(X')/\Br(k).
		\]
		
		Now consider the following exact sequence
		\[
		\begin{tikzcd}
			0 & \Br_{1}(X)/\Br(l) & \Br(X)/\Br(l) & \Br(X_{\overline{k}})^{\Gamma_{l}} & \homology^{2}(l,\Pic(X_{\overline{k}})) \\
			0 & \Br_{1}(X')/\Br(k) & \Br(X')/\Br(k) & \Br(X'_{\overline{k}})^{\Gamma_{k}} & \homology^{2}(k,\Pic(X'_{\overline{k}}))
			\arrow[from=1-1, to=1-2]
			\arrow[from=1-2, to=1-3]
			\arrow["\simeq"', from=1-2, to=2-2]
			\arrow[from=1-3, to=1-4]
			\arrow["\phi"', from=1-3, to=2-3]
			\arrow[from=1-4, to=1-5]
			\arrow["\simeq"', from=1-4, to=2-4]
			\arrow[from=2-1, to=2-2]
			\arrow[from=2-2, to=2-3]
			\arrow[from=2-3, to=2-4]
			\arrow[from=2-4, to=2-5]
			\arrow["\simeq"', from=1-5, to=2-5]
		\end{tikzcd}
		\]
		We have just seen that $\Br_{1}(X)/\Br(l)$ is isomorphic to $\Br_{1}(X')/\Br(k)$. Also, assertion (b) ensures that $\Br(X_{\overline{k}})^{\Gamma_{l}}\simeq \Br(X_{\overline{k}}')^{\Gamma_{k}}$ and $\homology^{2}(l,\Pic(X_{\overline{k}}))\simeq \homology^{2}(k,\Pic(X_{\overline{k}}'))$. Applying the five lemma, we conclude that the second vertical homomorphism \(\phi\) is also an isomorphism.
	\end{proof}
	
	\begin{remark}
		Note that over number fields, the assumption in Proposition~\ref{Proposition:WeirestricionBrauer}(c) is automatically satisfied.
	\end{remark}
		
When $l/k$ is an extension of number fields,
		the Weil restriction yields a natural identification
		\[
		\Phi: X(\A_{l}) \longrightarrow X'(\A_{k}).
		\]
		Using the compatibility of this map with the Brauer–Manin pairing in \cite[Section 2.4]{CL22}, one deduces equalities of Brauer--Manin sets.
		\begin{corollary}\label{Corollary:identityofWeil}
			Let \(l/k\) be an extension of number fields, and let \(X\) be a smooth, projective, geometrically integral variety over \(l\) such that \(\Pic(X_{\overline{k}})\) is  torsion‑free. Then \(\Phi\) induces  identifications
		\[
		X(\A_{l})^{\Br_1}=X'(\A_{k})^{\Br_1}\mbox{ and } X(\A_{l})^{\Br}=X'(\A_{k})^{\Br},
		\]
			where \(X' = \mathrm{Res}_{l/k}(X)\) is the Weil restriction of $X$. 
		
		\end{corollary}

Because  $\textup{Hom}_{\textup{cts}}(\pi_1^{\textup{ab}}(\overline{X}),\RR/\ZZ)=\Pic(\overline{X})_\textup{tor}$ for smooth projective varieties, cf. \cite[Remark 3.3]{CHto}, the preceding corollary recovers \cite[Theorem 1.3]{CHto} and the projective case of \cite[Theorem 1.2]{CHto}.  Their argument made use of descent theory while we compare relevant Brauer groups.

	\subsection{Weil restriction as a geometric quotient}
	\label{Section:Weilquotient}
	
	We recall how the Weil restriction of scalars can be viewed as a geometric quotient, which plays an important role in the forthcoming proof of the main theorems. We refer to \cite[Section 1.3]{Wei82} for the definition of Weil restrictions.
	
	Let \(X\) be a quasi-projective variety over a field \(l\), and let \(l/k\) be a separable field extension of degree \(d\). Choose a Galois closure \(K/k\) that contains \(l\) as a subfield.  Let  \(\{\sigma_s\}_{s=1}^d\) be the set of all \(k\)-embeddings of \(l\) into \(K\). Consider the $K$-variety
	\[
	Y = \prod_{s=1}^{d} X_{K,s},
	\]
	where each factor $X_{K,s}$ is the base change of \(X\) to \(K\) via the embedding \(\sigma_s\). A $k$-variety $W$ together with a $l$-morphism $q: W\times_{k}l \to X$ gives rise to a $K$-morphism
	\begin{equation*}
		q^{\sigma_s}:W_{K}\to X_{K,s},
	\end{equation*} which makes up a $\Gamma_{k}$-equivariant $K$-morphism
	\begin{equation*}
		\prod_{s}q^{\sigma_s}: W_{K}\to Y.
	\end{equation*}
	
	\begin{definition}[{\cite[Section 1.3]{Wei82}\label{definition:Weil}}]
		The Weil restriction \(X'\coloneqq \operatorname{Res}_{l/k}(X)\) of $X$ is a variety over \(k\) equipped with a morphism
		\[
		p: X' \times_k l \longrightarrow X
		\]
		such that the $\Gamma_{k}$-equivariant morphism
		\[
		\prod_{\sigma_s} p^{s}:X'_K \stackrel{\sim}{\longrightarrow} Y.
		\]
		is an isomorphism.
		
	\end{definition}
	
	In fact, the Galois group \(\operatorname{Gal}(K/k)\) acts naturally on the product \(Y\), and the morphism \(p\) provides a descent datum that exhibits \(X'\) as a Galois descent of \(Y\). Consequently, \(X'\) can be identified with the geometric quotient \(Y / \operatorname{Gal}(K/k)\) of $Y$ viewed as a $k$-variety.
	
	Note that the construction above is independent of the choice of the Galois closure \(K\).

	\section{Unramified Brauer groups of symmetric products}\label{Section:Symmetricproduct}
	Let \(n\) be a positive integer. The symmetric group \(S_n\) acts naturally on \(W^n\) by permuting the product factors. If \(W\) is a quasi‑projective variety over a field \(k\), the geometric quotient \(W^n/S_n\) exists as a variety together with a finite surjective canonical morphism
	\[
	W^n \longrightarrow W^n/S_n,
	\]  
cf. \cite[Chapter II, \S7 and Chapter III, \S12]{Mum08}. We define this quotient to be the $n$-fold \emph{symmetric product} of \(W\) and denote it by \(\operatorname{Sym}^n_{W}\).
	
Let $H_{i}=\{\sigma \in S_n \mid \sigma(i)=i\}\subset S_n$ be the stabilizer subgroup  with respect to $1\in\{1,2,\ldots,n\}$, on which $S_n$ acts by permutation. Then $H_i\simeq S_{n-1}$ and we obtain an isomorphism \begin{equation}\label{eq:nn-1}W^{n}/H_{i}\simeq W\times \Sym_{W}^{n-1}.\end{equation} 

	Now let \({W_0} \subseteq W^n\) be an \(S_n\)-stable open subset.
    As the restriction of $i$-th projection \(\operatorname{pr}_i :W^n \to W\) to ${W_0}$ is invariant under the action of $H_i$, it induces a morphism  ${W_0}/H_i \longrightarrow W$, which is nothing but the restriction to ${W_0}/H_i$ of the projection onto the $W$ factor via the isomorphism \eqref{eq:nn-1}.
	
	If we further assume that \(S_n\) acts \emph{freely} on \({W_0}\) and that \({W_0}\) is smooth, quasi‑projective and geometrically integral, then the quotient \({W_0}/S_n\) is smooth and the natural morphism  ${W_0}/H_i \longrightarrow {W_0}/S_n$
is finite \'etale.

	Let \(\Delta = \Delta_W^n\) denote the union of the diagonals  in \(W^n\) as defined in \cite[Section~4]{CZ24}. Then the symmetric group \(S_n\) acts freely on \({W_0}=W^n \backslash\Delta\). We denote the open subset \((W^n \backslash \Delta)/S_n\) of \(\operatorname{Sym}_{W}^n\) by  \(\operatorname{Sym}_{W}^{n,o}\). We will explain in Section \ref{Section:ratpt0cyc} the identification of the set of rational points of \(\operatorname{Sym}_{W}^n\) and the set of effective \(0\)-cycles on \(W\) of degree \(n\), where rational points of \(\operatorname{Sym}_{W}^{n,o}\) correspond to effective separable \(0\)-cycles.
	
	For each $n$, we fix a smooth projective model \(\operatorname{Sym}_{W}^{n,\mathrm{sm}}\) of \(\operatorname{Sym}_{W}^n\). Note that if the quotient \(\Brnr(k(\operatorname{Sym}_{W}^{n})/k)/\operatorname{Br}(k)\) is finite, then HPBMrp (respectively, WABMrp) does not depend on the choice of such a model. For HPBMrp this follows from the Lang–Nishimura theorem; for WABMrp we refer to \cite[Proposition 6.1]{CPS16}.
	
	\begin{lemma}
		If $W$ is geometrically reduced (resp. geometrically irreducible, geometrically integral, quasi-projective, projective) over k, then so is $\Sym_{W}^{n}$. If $W$ is smooth, then so is $\Sym_{W}^{n,o}$. 
	\end{lemma}

	We pass to a relative version. Let \(f: Y \to Z\) be a dominant morphism between smooth, quasi‑projective, geometrically integral varieties defined over \(k\). It naturally induces a morphism $f^{n}:Y^{n}\to Z^{n}$. Let \(U \subset Z^n\) be the open subset \(Z^n \setminus \Delta_Z^n\), and set \(V := (f^n)^{-1}(U)\). Then the symmetric group \(S_n\) acts freely on \(V\). Consequently, the induced morphism  
	\[
	f_{n}: \operatorname{Sym}_{Y}^n \longrightarrow \operatorname{Sym}_{Z}^n
	\]  
	satisfies  
	\[
	f_{n}^{-1}\bigl(\operatorname{Sym}_{Z}^{n,o}\bigr) \subset \operatorname{Sym}_{Y}^{n,o}.
	\]
	
	\begin{lemma}
		\label{Lemma:cartesianofquotient}
		Let \(f: Y \to Z\) be a dominant morphism between smooth, quasi-projective varieties over a field k, let \(U \subset Z\) be an open subset on which the symmetric group \(S_n\) acts freely. Set \(V := f^{-1}(U)\). 
Then we have a commutative diagram with  Cartesian squares
		\[
		\begin{tikzcd}
			V \arrow[r, "p_V"] \arrow[d] & V/H_1 \arrow[r, "q_V"] \arrow[d] & V/S_n \arrow[d] \\
			U \arrow[r, "p_U"] & U/H_1 \arrow[r, "q_U"] & U/S_n
		\end{tikzcd}
		\]
		where the vertical arrows are induced by \(f: Y \to Z\).
	\end{lemma}
	
	\begin{proof}
	Since $S_{n}$ acts on $U$ freely, we know that $U$ is a $U/S_{n}$-torsor under $S_{n}$, hence $U\times_{U/S_{n}} V/S_{n}$ is a $V/S_{n}$-torsor under $S_{n}$. At the same time $V$ is also a $V/S_{n}$-torsor under  $S_{n}$. The induced morphism from $V$ to $U\times_{U/S_{n}}V/S_{n}$ must be  an isomorphism $V\simeq U\times_{U/S_{n}} V/S_{n}$. We know that the big square is Cartesian, so is the left square. Then by the faithfully flat descent, we know that the right square is also Cartesian.	
	\end{proof}

	\subsection{Comparison homomorphisms between relevant Brauer groups}
	Consider the natural projection $\pr_{1}$ and the finite \'etale morphism $u$ for a smooth quasi-projective geometrically integral variety \(W\) over \(k\)
	\begin{equation*}
		W \xleftarrow{\pr_1} (W^n \setminus \Delta)/H_1 \xrightarrow{u} \Sym_{W}^{n,o} 
	\end{equation*}
and morphisms induced between the relevant function fields
	\begin{equation*}
		\Spec(k(W)) \xleftarrow{\pr_1} \Spec(k((W^n \setminus \Delta)/H_1)) \xrightarrow{u} \Spec(k(\Sym_{W}^{n,o})).
	\end{equation*}

Define $\lambda$ and $\lambda_\textup{nr}$ as the compositions
\begin{equation*}\label{lambda}\begin{tikzcd}
		\Br(W)& \Br((W^n \setminus \Delta)/H_1)) & \Br(\Sym_{W}^{n,o})
		\arrow["\textup{pr}_1^*",from=1-1, to=1-2]
		\arrow["\lambda", curve={height=-30pt}, from=1-1, to=1-3]
		\arrow["u_*", from=1-2, to=1-3]
	\end{tikzcd}\end{equation*}	
\begin{equation*}\label{lambda_nr}\begin{tikzcd}
		\Brnr(k(W)/k)& \Brnr(k((W^n \setminus \Delta)/H_1)/k) & \Brnr(k(\Sym_{W}^{n,o})/k)
		\arrow["\textup{pr}_1^*",from=1-1, to=1-2]
		\arrow["\lambda_{\textup{nr}}", curve={height=-30pt}, from=1-1, to=1-3]
		\arrow["u_*", from=1-2, to=1-3]
	\end{tikzcd}\end{equation*}
of restriction homomorphism $\textup{pr}_1^*$ and corestriction homomorphism $u_*$ between Brauer groups (\cite[Chapter 3.8]{CTS21}) and between unramified Brauer groups (see \cite[Proposition 6.2.3, 6.2.4]{CTS21}). 

For regular integral varieties, \cite[Theorem 3.7.3]{CTS21} ensures that the unramified Brauer group is naturally a subgroup of the Brauer group. When in addition the variety is proper, we can identify these two groups.  Hence, if $W$ is proper we have  \(\Br(W)= \Brnr(k(W)/k)\) and $\lambda$ factors through $\lambda_\textup{nr}$. We pass to quotients to obtain
	\begin{equation}
		\lambda : \Br(W)/\Br(k) \longrightarrow \Brnr\bigl(k(\Sym_{W}^{n,o})/k\bigr)/\Br(k). \label{*}\tag{$\star$}
	\end{equation}
 still denoted by \(\lambda\) by abuse of notation.
We will explain later in Remark \ref{remarklambda} the reason why we construct $\lambda$ in such a way.	

	\begin{proposition}[Functoriality of $\lambda$]	\label{Proposition: CommutativeforBrauer}
		Let \(Y\) and \(Z\) be smooth, projective, geometrically integral \(k\)-varieties, and let \(f: Y \to Z\) be a dominant morphism of \(k\)-varieties. Then the following diagram commutes
		
		\[
		\begin{tikzcd}
			\Br(Z) \arrow[r, "\lambda"] \arrow[d] & \Brnr(k(\operatorname{Sym}_{Z}^n)/k) \arrow[d] \\
			\Br(Y) \arrow[r, "\lambda"] & \Brnr(k(\operatorname{Sym}_{Y}^n)/k)
		\end{tikzcd}
		\]

	\end{proposition}
	\begin{proof}
		Applying \cite[Proposition 3.8.1]{CTS21} to the right square of Lemma~\ref{Lemma:cartesianofquotient}, we obtain the commutative diagram
		
		\[
		\begin{tikzcd}
			\Brnr(k(U/H_{1})/k) \arrow[r, "u_*"] \arrow[d] & \Brnr(k(\operatorname{Sym}_{Z}^n)/k) \arrow[d] \\
			\Brnr(k(V/H_{1})/k) \arrow[r, "u_*"] & \Brnr(k(\operatorname{Sym}_{Y}^n)/k)
		\end{tikzcd}
		\]
The desired square is obtained by composing the above square  with the restriction square induced by $\textup{pr}_1$.
	\end{proof}

	\subsection{Generic fibers of morphisms between symmetric products}
	Let \(f: Y \rightarrow Z\) be a dominant morphism between smooth quasi-projective and geometrically integral varieties defined over a field $k$.
 In this section, we aim to describe the generic fiber of the induced morphism  
	\[
	f_n: \operatorname{Sym}_{Y}^n \rightarrow \operatorname{Sym}_{Z}^n.
	\]

	Let \( \eta \) be the generic point of \( Z \) and \( \xi \) be the generic point of \( \operatorname{Sym}_{Z}^n \). Let \( \theta \) be the generic point of \( Z^n \). Since \( Z \) is geometrically integral, the point \( \theta \) is also the generic point of \( \operatorname{Spec}\bigl(k(\eta)\bigr)^{n} \). We then obtain the following diagram with Cartesian squares
	
	\begin{equation}\label{YnZngenericfiber}
	\begin{tikzcd}
		\prod\limits_{i=1}^{n}Y_{\eta}\otimes_{k(\eta),\pr_{i}}k(\theta) & Y_{\eta}^{n} & Y^{n} \\
		\operatorname{Spec}(k(\theta)) & \operatorname{Spec}\bigl(k(\eta)\bigr) ^{n}& Z^{n}.
		\arrow[from=1-1, to=1-2]
		\arrow[from=1-1, to=2-1]
		\arrow[from=1-2, to=1-3]
		\arrow[from=1-2, to=2-2]
		\arrow[from=1-3, to=2-3]
		\arrow[from=2-1, to=2-2]
		\arrow[from=2-2, to=2-3]
	\end{tikzcd}
	\end{equation}
		Here \( \pr_i \) denotes not only the $i$-th projection but also the \(i\)-th 
 embedding $\pr_i: k(\eta)\to k(\theta)$ of fields induced by the composition
	\[
	\operatorname{Spec}(k(\theta)) \rightarrow \operatorname{Spec}\bigl(k(\eta)\bigr) ^n\xrightarrow{\pr_i} \operatorname{Spec}(k(\eta)).
	\]
The symmetric group $S_n$ acts  on the diagram equivariantly by permuting the factors of the varieties.
	It follows by passing to quotients that the generic fiber of the morphism  
	\[
	f_n: \operatorname{Sym}_{Y}^n \rightarrow \operatorname{Sym}_{Z}^n
	\]  
	is 
	\[
	\Bigl(\prod\limits_{i=1}^{n} Y_{\eta} \otimes_{k(\eta),\pr_i} k(\theta)\Bigr) / S_n.
	\]

	Since \(\dim(Z) \geq 1\), the point \(\theta\) is contained in the separable locus of the morphism \(Z^n \to Z^n / S_n\). Consequently, \(k(\theta)\) is a Galois extension of \(k(\xi)\) of Galois group \(S_n\). Define \(L_i=k(\theta)^{H_i}\). Since the only normal subgroup of $S_{n}$ contained in $H_{i}$ is the trivial group, we know that $k(\theta)$ is the Galois closure of each \(L_i\) such that $\Gal(k(\theta)/L_i)=H_i$. It is straightforward to observe that $L_i = k(Z^n / H_i)$.
The composition of $$Z^n\to Z^{n}/H_i\simeq Z\times \Sym_{Z}^{n-1}$$ with the natural projection onto $Z$ leads to finite extensions $k(\eta)\subset L_i\subset k(\theta)$ of fields.
Set $L=L_1$. Let $\gamma_i$ be the transposition in $S_n=\operatorname{Gal}(k(\theta)/k(\xi))$ switching $1$ and $i$. Let $\iota_i:L\to k(\theta)$   be the restriction to $L$ of the $k(\xi)$-automorphism $\gamma_i:k(\theta)\to k(\theta)$. Then it maps $L$ isomorphically to $L_i$ by Galois theory and it extends   $\pr_i:k(\eta)\to k(\theta)$.

	Applying Lemma~\ref{Lemma:cartesianofquotient} to the morphism $\prod_{i=1}^{n}Y_{\eta}\otimes_{k(\eta),\pr_{i}}k(\theta)\rightarrow \Spec(k(\theta))$, we get the following Cartesian diagram
	\[
	\begin{tikzcd}[row sep=large, column sep=large]
		\bigl(\prod_{i=1}^n Y_\eta \otimes_{k(\eta), \pr_i} k(\theta)\bigr) / H_1 \arrow[r] \arrow[d] &
		\bigl(\prod_{i=1}^n Y_\eta \otimes_{k(\eta), \pr_i} k(\theta)\bigr) / S_n \arrow[d] \\
		\operatorname{Spec}(L) \arrow[r] &
		\operatorname{Spec}(k(\xi)).
	\end{tikzcd}
	\]
	In other words, we have an isomorphism
	\begin{equation}
	\bigl(\prod_{i=1}^n Y_\eta \otimes_{k(\eta), \pr_i} k(\theta)\bigr) / S_n \otimes_{k(\xi)}L\simeq \bigl(\prod_{i=1}^n Y_\eta \otimes_{k(\eta), \pr_i} k(\theta)\bigr) / H_1 .
	\label{Equation:basechange}
	\end{equation}
	
	As the projection to the first component
	\[
	\prod_{i=1}^{n} Y_{\eta} \otimes_{k(\eta), \pr_i} k(\theta) 
	\xrightarrow{} 
	Y_{\eta} \otimes_{k(\eta), \pr_1} k(\theta),
	\]
is  $H_1$-equivariant, by passing to quotients, it induces a $L$-morphism
	\begin{equation}\label{Equation:Lmorphism}
	\Bigl(\prod_{i=1}^{n} Y_{\eta} \otimes_{k(\eta), \pr_i} k(\theta)\Bigr) / H_1 
	\longrightarrow 
	Y_{\eta} \otimes_{k(\eta)} L.
	\end{equation}
	By base change via the $k(\xi)$-embedding $\iota_i:L\to k(\theta)$, we obtain
	\begin{equation*}
		\beta_i: 
		\Bigl(\prod_{i=1}^{n} Y_{\eta} \otimes_{k(\eta), \pr_i} k(\theta)\Bigr) / H_1 \otimes_{L,\iota_i} k(\theta)  \rightarrow Y_{\eta}\otimes_{k(\eta),\pr_i} k(\theta). 
	\end{equation*}
	According to \eqref{Equation:basechange}, the last morphism is nothing but
	\begin{equation*}
		\Bigl(\prod_{i=1}^{n} Y_{\eta} \otimes_{k(\eta), \pr_i} k(\theta)\Bigr) / S_{n}
		\otimes_{k(\xi)} k(\theta)\rightarrow Y_{\eta}\otimes_{k(\eta),\pr_i} k(\theta). 
	\end{equation*}
Defined component by component, we get a morphism
	\[
	\prod_{i=1}^n \beta_i : 
	\Bigl(\prod_{i=1}^{n} Y_{\eta} \otimes_{k(\eta), \pr_i} k(\theta)\Bigr) / S_{n}
	\otimes_{k(\xi)} k(\theta) 
	\longrightarrow 
	\prod_{i=1}^{n} Y_{\eta} \otimes_{k(\eta), \pr_i} k(\theta),
	\]
	which is  an isomorphism by Lemma~\ref{Lemma:cartesianofquotient} applied to  $\prod_{i=1}^{n}Y_{\eta}\otimes_{k(\eta),\pr_i}k(\theta)\rightarrow \Spec(k(\theta))$. Therefore, by Definition~\ref{definition:Weil}, 
	we conclude that 
	\begin{equation}\label{Equation:Weil}
	(\prod_{i=1}^{n}Y_{\eta}\otimes_{k(\eta),\pr_{i}}k(\theta))/S_{n} \simeq \Res_{L/k(\xi)}\bigl(Y_{\eta} \otimes_{k(\eta)} L\bigr),
\end{equation}
which is the generic fiber of
$f_{n}:\Sym_{Y}^{n}\rightarrow \Sym_{Z}^{n}$.

	\begin{proposition}\label{prop:commutativeBrnr}\label{Proposition:commutativeforBrnr}
Let \(f: Y \rightarrow Z\) be a dominant morphism between smooth quasi-projective and geometrically integral varieties defined over a field \(k\). Assume further that \(Y\) is projective. Then the comparison homomorphism $\lambda$ as defined in \eqref{*} fits into a commutative diagram where the bottom homomorphism between Brauer groups of relevant generic fibers is defined in \eqref{Definition:phi} 
		\[
		\begin{tikzcd}[row sep=large, column sep=large]
			\Br(Y) \arrow[r, "\lambda"] \arrow[d] &
			\Brnr\!\left(\Sym_{Y}^{n}\right) \arrow[d] \\
			\Br\!\left(Y_{\eta} \otimes_{k(\eta)} L\right) \arrow[r, "\phi"] &
			\Br\!\left(\Res_{L/k(\xi)}\left(Y_{\eta} \otimes_{k(\eta)} L\right)\right)
		\end{tikzcd}
		\]
		
	\end{proposition}
	
	\begin{proof}
		Let \(U \subset Z^{n}\) be a Zariski-open subset on which the symmetric group \(S_{n}\) acts freely. 
		Denote by \(V = f^{-1}(U)\) its preimage under \(f: Y^{n} \rightarrow Z^{n}\). It contains the the generic fiber of $f$, which is given by the $S_{n}$-equivariant Cartesian square \eqref{YnZngenericfiber}. In other words, we have an $S_{n}$-equivariant morphism
		  \begin{equation}\label{Inclusionofgeneric}
		  \prod_{i=1}^{n}Y_{\eta}\otimes_{k(\eta),\pr_{i}}k(\theta)\to V.
		 \end{equation} 
		 Applying Lemma~\ref{Lemma:cartesianofquotient}, thanks to \eqref{Equation:Weil} we obtain the Cartesian square
		\[
		\begin{tikzcd}[row sep=large, column sep=large]
			\bigl(\prod_{i=1}^n Y_\eta \otimes_{k(\eta), \pr_i} k(\theta)\bigr) / H_1  \arrow[r] \arrow[d] &
			\Res_{L/k(\xi)}\left(Y_{\eta} \otimes_{k(\eta)} L\right) \arrow[d] \\
			V/H_{1} \arrow[r] &
			V/S_{n}
		\end{tikzcd}
		\]
		
		Taking Brauer groups and using the corestriction homomorphism associated with the finite étale covers, we obtain the commutative diagram as in \cite[Proposition 3.8.1]{CTS21}
		\[
		\begin{tikzcd}[row sep=large, column sep=large]
			\Br(V/H_{1}) \arrow[r, "\mathrm{cor}=u_*"] \arrow[d] &
			\Br(V/S_{n}) \arrow[d] \\
			\Br\!\left(\bigl(\prod_{i=1}^n Y_\eta \otimes_{k(\eta), \pr_i} k(\theta)\bigr) / H_1 \right) 
			\arrow[r, "\mathrm{cor}"] &
			\Br\!\left(\Res_{L/k(\xi)}\left(Y_{\eta} \otimes_{k(\eta)} L\right)\right)
		\end{tikzcd}
		\]
		
		Since all varieties involved are regular and integral, the unramified Brauer group \(\Brnr(-)\) is a natural subgroup of \(\Br(-)\). 
		By functoriality of the unramified Brauer group, we may restrict the top row of the previous diagram to obtain
		\begin{equation}
		\begin{tikzcd}[row sep=large, column sep=large]
			\Brnr(V/H_{1}) \arrow[r] \arrow[d] &
			\Brnr(V/S_{n}) \arrow[d] \\
			\Br\!\left(\bigl(\prod_{i=1}^n Y_\eta \otimes_{k(\eta), \pr_i} k(\theta)\bigr) / H_1 \right) 
			\arrow[r] &
			\Br\!\left(\Res_{L/k(\xi)}\left(Y_{\eta} \otimes_{k(\eta)} L\right)\right)
		\end{tikzcd}
		\label{Diagram:1}
		\end{equation}
		
		On the other hand, we have a natural commutative diagram where the homomorphism on the top is \eqref{Equation:Lmorphism} and the left vertical homomorphism is induced by \eqref{Inclusionofgeneric}
		\[
		\begin{tikzcd}[row sep=large, column sep=large]
			\bigl(\prod_{i=1}^n Y_\eta \otimes_{k(\eta), \pr_i} k(\theta)\bigr) / H_1  \arrow[r] \arrow[d] &
			Y_{\eta} \otimes_{k(\eta)} L \arrow[d] \\
			V/H_{1} \arrow[r,"\pr_1"] &
			Y
		\end{tikzcd}
		\]
		which, by applying the Brauer group functor, gives
		\[
		\begin{tikzcd}[row sep=large, column sep=large]
			\Br(Y) \arrow[r,"\pr_1^*"] \arrow[d] &
			\Br(V/H_{1}) \arrow[d] \\
			\Br\!\left(Y_{\eta} \otimes_{k(\eta)} L\right) \arrow[r] &
			\Br\!\left(\bigl(\prod_{i=1}^n Y_\eta \otimes_{k(\eta), \pr_i} k(\theta)\bigr) / H_1 \right).
		\end{tikzcd}
		\]
		
As $Y$ is assumed projective, the image of top row lies in the unramified Brauer group, we obtain
		\begin{equation}
		\begin{tikzcd}[row sep=large, column sep=large]
			\Br(Y) \arrow[r,"\pr_1^*"] \arrow[d] &
			\Brnr(V/H_{1}) \arrow[d] \\
			\Br\!\left(Y_{\eta} \otimes_{k(\eta)} L\right) \arrow[r] &
			\Br\!\left(\bigl(\prod_{i=1}^n Y_\eta \otimes_{k(\eta), \pr_i} k(\theta)\bigr) / H_1 \right).
		\end{tikzcd}
		\label{Diagram2}
		\end{equation}
		
		We complete the proof by composing the diagrams (\ref{Diagram:1}) and (\ref{Diagram2}). We should remark that, the quotient $\prod_{i=1}^{n}Y_{\eta}\otimes_{k(\eta),\pr_{i}}k(\theta)/H_{1} $ and $\Res_{L/k(\xi)}\left(Y_{\eta} \otimes_{k(\eta)} L\right) \otimes_{k(\xi)} L$ are identified to each other by  \eqref{Equation:basechange} and \eqref{Equation:Weil}. Hence the composition of the bottom homomorphisms is exactly $\phi$  recalled in \eqref{Definition:phi}.
\end{proof}

	\subsection{Unramified Brauer groups for symmetric products}

	\begin{theorem}\label{theorem}
		Let \(X/k\) be a smooth, projective, geometrically integral variety such that 
		\(\Pic(X_{\overline{k}})\) is a torsion‑free abelian group. Assume, moreover, 
		that the homomorphism \(\epsilon_{X}\) in \eqref{epsilon} is an isomorphism. Then the homomorphism 
		\[
		\lambda \colon \Br(X)/\Br(k) \xrightarrow{\ \simeq\ } 
		\Brnr \!\bigl(\Sym_{X}^{n}\bigr)/\Br(k)
		\]
 as defined  in \eqref{*} is an isomorphism.
	\end{theorem}
	
	\begin{proof}
		Set \(Y := X \times \PP^{1}\) and consider the natural projection \(\pr \colon Y \to \PP^{1}\). To this trivial fibration, we apply 
		Proposition~\ref{Proposition:commutativeforBrnr} and obtain a commutative diagram
		\begin{equation}\label{Diagram:main}
		\begin{tikzcd}[row sep=large, column sep=large]
			\Br(Y)/\Br(k) \arrow[r, "\lambda"] \arrow[d] &
			\Brnr\!\bigl(\Sym_{Y}^{n}\bigr)/\Br(k) \arrow[d] \\
			\Br(X \otimes_{k} L)/\Br(L) \arrow[r, "\phi"] &
			\Br\!\bigl(\Res_{L/k(\xi)}(X \otimes_{k} L)\bigr)/\Br(k(\xi))
		\end{tikzcd}
		\end{equation}
		where \(\xi\) denotes the generic point of \(\Sym_{\PP^{1}}^{n}\).  
		
		Because \(X/k\) is geometrically integral, the extension \(L/k\) is regular. 
		Hence, by Lemma~\ref{Lemma:Brauerregularextension}(c),
		\[
		\Br(X)/\Br(k) \simeq \Br(X_{L})/\Br(L),
		\]
		so the left vertical arrow is an isomorphism. The Lemma~\ref{Lemma:Brauerregularextension}(c) also tells us that 
		\(\epsilon_{X_{L}}\) is an isomorphism. Notice that Lemma~\ref{Lemma:torsionfreepicard} ensures that $\Pic(X_{\overline{L}})$ is also torsion-free, therefore Proposition~\ref{Proposition:WeirestricionBrauer}(c) 
		implies that the bottom  arrow \(\phi\) is an isomorphism.
		
		By \cite[Lemma 2.4]{Kol18}, the symmetric products \(\Sym_{X}^{n}\) 
		and \(\Sym_{Y}^{n}\) are stably birational equivalent.  If we can show that the right vertical arrow is injective, then the top horizontal 
		arrow \(\lambda\) must also be an isomorphism. Combining this with the stable 
		birational equivalence and Proposition~\ref{Proposition: CommutativeforBrauer} applied to $\textup{pr}_X:Y=X \times \PP^{1}\to X$, the top row of Diagram \eqref{Diagram:main} reduces to an isomorphism
		\[
		\lambda \colon \Br(X)/\Br(k) \xrightarrow{\ \simeq\ } 
		\Brnr\!\bigl(\Sym_{X}^{n}\bigr)/\Br(k)
		\]
		as defined  in \eqref{*}.

		It remains to prove the injectivity of the right vertical homomorphism in Diagram \eqref{Diagram:main}.  
		Consider the natural projection \(\pr \colon Y \to \PP^{1}\). 
        By Hironaka's resolution of singularity, we have a smooth projective geometrically integral model $W$ of $\Sym_{Y}^{n}$ such that the rational map $\hat{f}_n:W\dashrightarrow \PP^n$ defined via $f_n:\Sym_{Y}^{n}\to\Sym_{\PP^{1}}^{n}=\PP^{n}$ is a morphism.     According to \cite[Lemma 3.3]{CZ24}, the fiber of $\hat{f}_n$ over any codimension \(1\) point of \(\PP^{n}\) is split.
		This yields a commutative diagram of complexes whose bottom row is exact, cf. \cite[Diagram (11.3)]{CTS21}
		\begin{equation*}
				\begin{tikzcd}[row sep=large, column sep=small]
				0 \arrow[r] & 
				\Br(W) \arrow[r] & 
				\Br(W_{\xi}) \arrow[r] & 
				\displaystyle\bigoplus_{P\in\PP^{n,(1)}}\bigoplus_{V\subset W_P}
				\homology^{1}\bigl(k(V),\bQ/\bZ\bigr) \\
				0 \arrow[r] & 
				\Br(\PP^{n}) \arrow[r] \arrow[u, "\hat{f}^{*}_n"] & 
				\Br(k(\xi)) \arrow[r] \arrow[u, "\hat{f}_{n\xi}^{*}"] & 
				\displaystyle\bigoplus_{P\in\PP^{n,(1)}}
				\homology^{1}\bigl(k(P),\bQ/\bZ\bigr) \arrow[u, "{e_V\cdot\textup{res}}"']
			\end{tikzcd}
			\label{Commutetativediagram}
		\end{equation*}

		It follows from  \cite[Lemma 11.1.2]{CTS21} and diagram chasing that 
		\begin{equation*}
			\Br(W)/\Br(\PP^n)\hookrightarrow \Br(W_{\xi})/\Br(k(\xi))
		\end{equation*} 
        is a monomorphism. The Weil restriction $\Res_{L/k(\xi)}(X \otimes_{k} L)$, as the generic fiber of $f_n$ according to Formula \eqref{Equation:Weil}, is birational equivalent to the generic fiber $W_\xi$ of $\hat{f}_n$. Then the last monomorphism becomes 
        $$\Brnr(\Sym_{Y}^{n})/\Br(k)\to\Br(\Res_{L/k(\xi)}(X \otimes_{k} L))/\Br(k(\xi))$$ 
        which is the right vertical arrow in Diagram \eqref{Diagram:main}.
	\end{proof}

	\section{Hasse principle and weak approximation}\label{Section:Brauermanin}
	
\subsection{Effective 0-cycles and rational points of symmetric products}\label{Section:ratpt0cyc}

It is well-known that there is a bijection between the set of $k$-rational points of the symmetric product $\Sym_{X}^{n}$ and the set of  degree $n$ effective 0-cycles on $X$. In \cite[paragraphs before Proposition 2.2]{Fogarty}, John Fogarty explained this bijection in a formal way that we are going to reinterpret as follows.

 Recall that we  have a couple of morphisms
	\begin{equation*}
		X \xleftarrow{\pr_1} X^n /H_1 \xrightarrow{u} \Sym_{X}^{n} \label{correspondandceofpointsandcycles}
	\end{equation*}

With the help of these morphisms, we claim that the bijection concerned  can be written as
\begin{align*}
    \Lambda:\Sym_{X}^{n}(k)&\to Z_{0}^{n,\textup{eff}}(X) \\
    p&\mapsto \pr_{1*}(u^{*}(p)).
\end{align*}

It seems that the construction of  $\Lambda$ destroys the symmetry of the bijection in our mind, but we are going to explain in detail that it matches our impression of this bijection.

To simplify the notation, let us look at the simplest non-trivial case with $n=3$ for demonstration, while the statements will be written in full generality. 
The surjection 
$$X(\overline{k})\times X(\overline{k})\times X(\overline{k})\to X(\overline{k})\times X(\overline{k})\times X(\overline{k})/S_3=\Sym_{X}^{n}(\overline{k})$$
maps a $\overline{k}$-rational point $(p_1,p_2,p_3)$ to its $S_3$-orbit
$$S_3(p_1,p_2,p_3)=\{(p_1,p_2,p_3),(p_1,p_3,p_2),(p_2,p_1,p_3),(p_2,p_3,p_1),(p_3,p_1,p_2),(p_3,p_2,p_1)\}$$
viewed as a $\overline{k}$-rational point of $\Sym_{X}^{n}$. We should remark that elements in the orbit may not necessarily be distinct, which happens when $p_i=p_j$ for some certain $i\neq j$. We call it inseparable if such situation happens.
The absolute Galois group $\Gamma=\textup{Gal}(\overline{k}/k)$ acts equivariantly on both sides of the surjection. In other words, the $\Gamma$-action and the $S_n$-action on $X^n(\overline{k})$ are compatible in the sense that for any $\gamma\in\Gamma$ and any $\sigma\in S_n$ we have $\gamma\sigma(p_1,p_2,p_3)=\sigma\gamma(p_1,p_2,p_3)$. This makes the set of $S_n$-orbits (i.e. the set $\Sym_{X}^{n}(\overline{k})$ of $\overline{k}$-rational points) decompose as a disjoint union of $\Gamma$-orbits.
The $S_3$-orbit $S_3(p_1,p_2,p_3)$ represents a $k$-rational point $p$ of $\Sym_{X}^{3}$ if and only if it is fixed by the $\Gamma$-action.
In general, each $\Gamma$-orbit of $\Sym_{X}^{n}(\overline{k})$ is a closed point whose degree equals to the number of elements of the orbit. Similar situation happens on $(X^n/H_1)(\overline{k})$ and $X(\overline{k})$ as well.

Now $H_1=\langle\tau\rangle\simeq S_{n-1}$ is generated by the transposition $\tau=(23)\in S_3$, which acts on $X(\overline{k})\times X(\overline{k})\times X(\overline{k})$ by exchanging the 2nd and the 3rd coordinate of each triple. 

We start from a $k$-rational point $p\in \Sym_{X}^{n}(k)$ regarded as a $\overline{k}$-rational point and a closed point as above, i.e. an $S_3$-orbit $ S_3(p_1,p_2,p_3)$ fixed by $\Gamma$.
The inverse image $u^{-1}(p)=\{q_1,q_2,q_3\}$ is a set of $\overline{k}$-rational points of $X^n/H_1$. Written in the form of $H_1$-orbits, we have  
\begin{align*}
    q_1&=\{(p_1,p_2,p_3),(p_1,p_3,p_2)\}=H_1(p_1,p_2,p_3)=H_1(p_1,p_3,p_2), \\
    q_2&=\{(p_2,p_1,p_3),(p_2,p_3,p_1)\}=H_1(p_2,p_1,p_3)=H_1(p_2,p_3,p_1), \\
    q_3&=\{(p_3,p_1,p_2),(p_3,p_2,p_1)\}=H_1(p_3,p_1,p_2)=H_1(p_3,p_2,p_1).
\end{align*}
Some of them may coincide to each other in the inseparable situation.
From the list, we understand that the elements in each $H_1$-orbit $q_i$ is determined by their 1st coordinates $p_i$, while the other $n-1$ coordinates are obtained via all possible permutations of $(p_1,\dots,p_{i-1},p_{i+1},\ldots,p_n)$ with $p_i$ omitted.
In other words, $p_i$ completely determines $q_i$. And the 1st coordinate of $q_i$ is $p_i$, which does not depend on the choice of representatives of the associated $H_1$-orbit. 
This is still the case for the inseparable situation.

Now we consider $\Gamma$-actions on these elements. Since $p=S_3(p_1,p_2,p_3)$ is a $k$-rational point of $\Sym_{X}^{n}$, the $\Gamma$-stable set $u^{-1}(p)=\{q_1,q_2,q_3\}$ decomposes as a disjoint union of $\Gamma$-orbits. Each $\Gamma$-orbit represents a closed point of $X^n/H_1$. As a 0-cycle, $u^{*}(p)$ is the formal sum of these closed points
$$u^*(p)=\sum_{\Gamma q_i\in |X^n/H_1|} m(\Gamma q_i)\Gamma q_i,$$
where the second sum is taken over the set $|X^n/H_1|$ of closed points. 
Here the multiple 
\begin{align*}
   m(\Gamma q_i)&=|\{j\in\ZZ~|~1\leq j\leq n\mbox{ and }~q_j=q_i\in (X^n/H_1)(\overline{k})\}| \\
     &=|\{j\in\ZZ~|~1\leq j\leq n\mbox{ and }~p_j=p_i\in X(\overline{k})\}|,
\end{align*}
where the second equality follows from the fact explained above that $q_i$ and $p_i$ determine each other. This multiple is well-defined, i.e. it does not depend on the choice of representatives in the $\Gamma$-orbit $\Gamma q_i$, because  $\Gamma$ acts component by component. It also follows that  the 1st coordinates of the elements in the $\Gamma$-orbit $\Gamma q_i$ of $(X^n/H_1)(\overline{k})$ form exactly the $\Gamma$-orbit $\Gamma p_i$  of $X(\overline{k})$, which represents a closed point of $X$. So the image under $\pr_1$ of the closed point $\Gamma q_i$ is $\Gamma p_i$. 
Moreover, $\Gamma q_i$ and $\Gamma p_i$ have the same residue field since the stabilizer subgroups in $\Gamma$ of the $\overline{k}$-rational points $q_i$ and $p_i$ are equal. 
Therefore the 0-cycle 
$$\Lambda(p)=\pr_{1*}(u^{*}(p))
=\sum_{\Gamma p_i\in |X|} m(\Gamma q_i)\Gamma p_i$$
is exactly the one in our mind given by collecting the closed points corresponding to those $p_i$.
The 0-cycle $\Lambda(p)$ is effective of degree 
$$\sum_{\Gamma p_i\in |X|} m(\Gamma q_i)[k(\Gamma p_i):k]=\sum_{\Gamma p_i\in |X|} m(\Gamma q_i)|\Gamma p_i|=\sum_{p_i\in X(\overline{k})}m(\Gamma q_i)\cdot1=\sum_{i=1}^n 1=n.$$

\begin{remark}\label{remarklambda}
The construction of $\Lambda$ does play a key role in our paper. First of all, it makes the correspondence between $\Sym_{X}^{n}(k)$ and $Z_{0}^{n,\textup{eff}}(X)$  structured. Then it leads us to define the comparison homomorphism $\lambda$ in \eqref{*} between Brauer groups $\Br(X)$ and $\Brnr(\Sym_{X}^{n})$. It is the most reasonable way, for the purpose of studying Brauer--Manin obstructions, to define $\lambda$ as the composition of $\pr_1^*$ and the corestriction $u_*$ comparing to $\Lambda=\pr_{1*}\circ u^*$. 
\end{remark}

When the $p_i$'s are all distinct, the corresponding 0-cycle is separable with multiples $m(\Gamma q_i)=1$, 
and the formulae is simplified a lot. In the following applications, only the separable situation appears.

\subsection{Arithmetic applications to the local-global principle}

	We are ready to apply the results in Section \ref{Section:Symmetricproduct} on the unramified Brauer group of symmetric products to obtain arithmetic results.
	
	\begin{theorem}\label{thm:main}
		Let \(k\) be a number field and let \(X\) be a smooth, projective, geometrically integral variety over \(k\) such that \(\Pic(X_{\overline{k}})\) is a torsion‑free abelian group. Then the following hold.
		\begin{enumerate}[(a)]
			\item Under the identification
			\[
			\mathrm{Z}_{0,\A}^{n,\mathrm{sep},\mathrm{eff}}(X)=\prod_{v\in\Omega_k}\Sym_{X}^{n,o}(k_v)
			\]
			we have an equality of Brauer--Manin subsets
			\[
			\mathrm{Z}_{0,\A}^{n,\mathrm{sep},\mathrm{eff}}(X)^{\Br}= 
			\Bigl(\prod_{v\in\Omega_k}\Sym_{X}^{n,o}(k_v)\Bigr)^{\Brnr}.
			\]
			\item The variety \(X\) satisfies $\text{HPBM0cyc}^{n,\mathrm{eff,sep}}$ if and only if \(\Sym_{X}^{n,o}\) satisfies HPBMrp with respect to the unramified Brauer group.
			\item If \(\Sym_{X}^{n,o}\) satisfies the WABMrp, then \(X\) satisfies the $\text{WABM0cyc}^{n,\mathrm{eff,sep}}$.
			\item Assume moreover that \(\Br(X)/\Br(k)\) is finite. 
If $\Sym_{X}^{n,\mathrm{sm}}$ satisfies HPBMrp (resp. WABMrp) with respect to the unramified Brauer group for $n$ sufficiently large, then \(X\) satisfies $\text{HPBM0cyc}^{1}$(resp. $\text{WABM0cyc}^{1}$).
		\end{enumerate}
		
	\end{theorem}
	
	\begin{proof}
		
	\begin{enumerate}[(a)]
			\item Consider the couple of  morphisms
		\[
		X \xleftarrow{\mathrm{pr}_1} (X\setminus\Delta)/H_1 \xrightarrow{u} \Sym_{X}^{n,o}.
		\]
Take a family
			\(\{z_v\}_{v\in\Omega_k}\) of degree \(n\) effective separable local \(0\)-cycles.
			Each \(z_v\) can be written as \(\mathrm{pr}_{1*}(u^*(p_v))\) for a unique
			\(p_v \in \Sym_{X}^{n,o}(k_v)\). According to Theorem~\ref{theorem}, any element
			\(b' \in \Brnr(\Sym_{X}^{n})\) is written as a sum 
			\(b' = \lambda(b) + b_0\) for some certain \(b_0 \in \Br(k)\) and \(b \in \Br(X)\). Then
			\[
			\sum_{v\in\Omega_k} \langle p_v, b' \rangle
			= \sum_{v\in\Omega_k} \langle p_v, \lambda(b)+b_0 \rangle
			=\sum_{v\in \Omega_{k}} \langle p_{v}, u_{*}\pr_1^{*}(b)\rangle
			\]
			Because the corestriction homomorphism is compatible with pull-back \cite[Proposition 3.8.1]{CTS21}, we have
			\begin{equation*}
				\sum_{v\in \Omega_{k}}\langle p_{v}, u_{*}\pr_1^{*}(b)\rangle =\sum_{v \in \Omega_k} \langle u^{*}(p_{v}),\pr_1^{*}(b)\rangle
			\end{equation*}
			Then by functoriality, we have
			\begin{equation*}
				\sum_{v \in \Omega_k} \langle u^{*}(p_{v}),\pr_1^{*}(b)\rangle = \sum_{v \in \Omega_k}\langle \pr_{1*}u^{*}(p_{v}),b\rangle =\sum_{v \in \Omega_k}\langle z_{v},b \rangle 
			\end{equation*}
			Whence 
			\begin{equation*}
				\sum_{v \in \Omega_k} \langle p_{v},b' \rangle =\sum_{v \in \Omega_k} \langle z_{v},b\rangle
			\end{equation*}
			and thus the equality
			\[
			\mathrm{Z}_{0,\A}^{n,\mathrm{sep},\mathrm{eff}}(X)^{\Br}= 
			\Bigl(\prod_{v\in\Omega_k}\Sym_{X}^{n,o}(k_v)\Bigr)^{\Brnr}.
			\]
			
			\item It follows directly from (a).
			
			\item Let \(\{z_v\}\) be a family of separable effective \(0\)-cycles of degree
			\(n\) orthogonal to \(\Br(X)\). Let \(S\subset\Omega_k\) be a finite set and
			\(\delta\) a positive integer. Let \(\{p_v\}\) be the corresponding local
			points on \(\Sym_{X}^{n,o}\). For any \(b'\in\Brnr(\Sym_{X}^{n})\) written as
			\(b'=\lambda(b)+b_0\) as before, we have
			\[
			\sum_{v\in\Omega_k} \langle p_v, b'\rangle
			= \sum_{v\in\Omega_k} \langle z_v, b \rangle = 0.
			\]
			Hence \(\{p_v\}\) is orthogonal to \(\Brnr(\Sym_{X}^{n})\). By the hypothesis that
			\(\Sym_{X}^{n,o}\) satisfies WABMrp, we can find a global point
			\(p\in\Sym_{X}^{n,o}(k)\) arbitrarily close to each \(p_v\) for \(v\in S\).
			Let \(z\) be the separable effective \(0\)-cycle of degree \(n\) on \(X\)
			corresponding to \(p\). According to \cite[Lemma 1.8]{Wit12}, after choosing
			\(p\) sufficiently close to the local points, the cycle \(z\) has the same
			image in \(\CH_0(X)/\delta\) as \(z_v\) for every \(v\in S\). Thus \(X\)
			satisfies $\text{WABM0cyc}^{n,\mathrm{eff,sep}}$.
			
			\item Under the additional hypothesis, the Brauer group \(\Brnr(\Sym_{X}^{n})/\Br(k)\) is finite.
			Following the argument of \cite[Theorem 3.2.1]{Lia13}, it suffices to prove
			$\text{HPBM0cyc}^1$ (resp. $\text{WABM0cyc}^{1}$) for the variety \(X\times\PP^1\). Let \(\{z_v\}\) be a
			family of \(0\)-cycles of degree \(1\) on \(X\times\PP^1\) orthogonal to
			\(\Br(X\times\PP^1)\).
			
			By the Lang–Weil estimates we may enlarge \(S\) so that for every
			\(v\notin S\) the variety \(X_{k_v}\) possesses a \(k_v\)-point \(x_v\).
			Let \(\{b_i\}_{1\le i\le s}\)
			be a set of generators of the finite group \(\Br(X\times\PP^1)/\Br(k)\). These generators are torsion elements because they
			come from a smooth projective variety. Choose a positive integer \(m\) that
			annihilates all of them. Enlarge \(S\) further so that for every \(b_i\) and
			every \(v\notin S\) we have \(\langle z_v, b_i\rangle = 0\). Fix a global
			point $P$ of $X\times \PP^{1}$ and write \(\delta_P = [k(P):k]\).
			
			Fix a positive integer \(\delta\). According to \cite[Theorem 3.2.1]{Lia13}
			we can find, for each \(v\in S\), an effective separable \(0\)-cycle
			\(z_v^1\) of degree \(\Delta\) (a same integer for all \(v\in S\)) such that
			\(\Delta\equiv1\pmod{m\delta\delta_P}\), the cycles \(z_v\) and \(z_v^1\) have
			the same image in \(\CH_0(X\times\PP^1)/\delta\), and
			\(\langle z_v,b_i\rangle = \langle z_v^1,b_i\rangle\) for every \(i\).
			For places \(v\notin S\), the existence of a \(k_v\)-point on \(X_v\) and the
			fact that \(\PP^1_{k_v}\) admits separable effective \(0\)-cycles of any degree
			allow us to choose an effective separable \(0\)-cycle \(z_v^1\) on
			\(X\times\PP^1\) of the same degree \(\Delta\) with $\Delta$ sufficently large. The family
			\(\{z_v^1\}_{v\in\Omega_k}\) therefore consists of effective separable
			\(0\)-cycles of constant degree \(\Delta\). It corresponds to a family of local points
			\(\{p_v\}\) of \(\Sym_{X\times\PP^1}^{\Delta,o}\) that is orthogonal to
			\(\Brnr(\Sym_{X\times\PP^1}^{\Delta})\).
			
			Moreover, because \(\Sym_{X}^{\Delta,\mathrm{sm}}\) satisfies HPBMrp (resp. WABMrp)
			by assumption, the product \(\Sym_{X}^{\Delta,\mathrm{sm}}\times\PP^{\Delta}\) also
			satisfies HPBMrp (resp. WABMrp). By stable birational equivalence (\cite[Lemma 2.4]{Kol18}) and the finiteness of
			relevant Brauer groups, the same property holds for
			\(\Sym_{X\times\PP^1}^{\Delta,\mathrm{sm}}\).
			
			For the Hasse principle, the existence local points \(\{p_v\}\) on
			\(\Sym_{X\times\PP^{1}}^{\Delta,\mathrm{sm}}\) gives a \(k\)-rational point on $\Sym_{X\times \PP^{1}}^{\Delta}$ which corresponds to a
			global effective \(0\)-cycle \(z\) of degree \(\Delta\) on
			\(X\times\PP^1\). Since \(\Delta\equiv1\pmod{m\delta\delta_P}\), a suitable
			linear combination of \(z\) and \(m\delta\) copies of the cycle $P$
			yields a global \(0\)-cycle of degree \(1\).
			
			For weak approximation, we are given local points of
			\(\prod_{v\in\Omega_k}\Sym_{X\times\PP^1}^{\Delta,\mathrm{sm}}(k_{v})\) that lies in \(\prod_{v\in\Omega_k}\Sym_{X\times\PP^1}^{\Delta,o}(k_{v})\). As in \cite[Proposition 6.1]{CPS16}, we can
			approximate \(\{p_v\}\) by a \(k\)-point \(p\) in \(\Sym_{X\times\PP^1}^{\Delta,o}\)
			such that, by \cite[Lemma 1.8]{Wit12}, the corresponding global separable
			effective \(0\)-cycle \(z\) has the same image in \(\CH_0(X\times\PP^1)/\delta\)
			as \(z_v\) for every \(v\in S\). A suitable linear combination of \(z\) and
			\(\delta\) copies of the cycle $\delta P$ then gives a \(0\)-cycle of degree
			\(1\) that has the same image in \(\CH_0(X\times\PP^1)/\delta\) as the original
			local cycles \(z_v\) for all \(v\in S\).
		\end{enumerate}
	\end{proof}
	\begin{remark}
		Under the identification of Brauer--Manin subsets in (a), the relation of the HPBM between the 0-cycles and the rational points are deduced directly. But for the WABM, we also need to compare topologies of both sides of the bijection 
$$\Lambda: \prod_{v\in\Omega_k}\Sym_{X}^{n}(k_v)\to Z_{0,\A_k}^{n,\textup{eff}}(X)\subset Z_{0,\A_k}^{n}(X).$$
In a recent preprint \cite[Section 5]{HZ25} of the third author, he reformulates the weak approximation property for 0-cycles as a topology on the group of adelic 0-cycles. The lemma  of Olivier Wittenberg \cite[Lemma 1.8]{Wit12} appeared in our proof of (c) asserts  that $\Lambda$ is continuous, or equivalently speaking the product topology on the left is not coarser than that on the right.

	\end{remark}
	
	As another arithmetic applications, we are going to recover a result \cite[Theorem 1.3]{CZ24} of Chen and Zhang, with slightly weaker hypotheses. Before stating the theorem, we recall some definitions and lemmas from \cite{Lia16}.
	
	\begin{definition}[\text{\cite[Definition 3.3]{Lia16}}]
		Let \(d\) be a positive integer and \(X\) be a quasi-projective variety defined over a topological field \(k\). Effective \(0\)-cycles \(z\) and \(z'\) of degree \(d\) on \(X\) correspond to rational points \([z],[z']\in\Sym^{d}_X(k)\). One says that \(z\) is \emph{sufficiently close} to \(z'\) if \([z]\) is sufficiently close to \([z']\) with respect to the topology on \(\Sym^{d}_X(k)\) induced by the topology of \(k\).
	\end{definition}
	
	\begin{lemma}[Relative moving lemma, \text{\cite[Lemma~3.4]{Lia16}}] \label{moving lemma}
		Let \(\pi:X\to\mathbb{P}^1\) be a dominant morphism of algebraic varieties defined over \(\mathbb{R}\), \(\mathbb{C}\), or any finite extension of \(\mathbb{Q}_p\). Assume that \(X\) is smooth.
		Then for any effective \(0\)-cycle \(z'\) on \(X\), there exists an effective \(0\)-cycle \(z\) on \(X\) such that \(\pi_*(z)\) is separable and such that \(z\) is sufficiently close to \(z'\).
	\end{lemma}
	
	\begin{lemma}[Hilbert's irreducibility theorem for \(0\)-cycles, \text{\cite[Lemma~3.5]{Lia16}}]\label{Hilbert_irred}
		Let \(S\) be a nonempty finite set of places of a number field \(k\). Let \(\Hil\subset \mathbb{A}^{1}= \PP^{1}\setminus\{\infty\}\) be a generalized Hilbertian subset of \(\PP^{1}\). For each \(v\in S\), let \(z_{v}\) be a separable effective \(0\)-cycle of degree \(d>0\) with support contained in \(\mathbb{A}^{1}\).
		
		Then there exists a closed point \(\theta\in\PP^{1}\) such that
		\begin{itemize}
			\item \(\theta\in\Hil\);
			\item as a \(0\)-cycle, \(\theta\) is sufficiently close to \(z_{v}\) for every \(v\in S\).
		\end{itemize}
	\end{lemma}
	
	\begin{theorem}\label{thm:generalization}
		Let \(X\) be a smooth, projective, geometrically integral variety over a number field \(k\) such that \(\Pic(X_{\overline{k}})\) is torsion‑free and \(\Br(X)/\Br(k)\) is finite. Assume that there is a finite field extension $L/k$ such that for all the finite field extension $K/k$ linearly disjoint to $L$, denote by $X_{K}$ the base change of $X$ through the finite field extension $K/k$, the natural homomorphism
		\[
		\Br(X)/\Br(k) \longrightarrow \Br(X_{K})/\Br(K)
		\]
		is surjective, and that \(X_{K}\) satisfies HPBMrp (resp.\ WABMrp).
		
		Then for every positive integer \(n\), any smooth projective model of the symmetric product \(\Sym_{X}^{n,\mathrm{sm}}\) satisfies HPBMrp (resp.\ WABMrp).
	\end{theorem}

	\begin{proof}
		First, since \(\Br\!\bigl(\Sym_{X}^{n,\mathrm{sm}}\bigr)/\Br(k)\) is finite, the property that every smooth projective model of \(\Sym_{X}^{n}\) satisfies HPBMrp (resp.\ WABMrp) is equivalent to the existence of at least one smooth projective model of \(\Sym_{X}^{n}\) that satisfies HPBMrp (resp.\ WABMrp). We choose the particular model \(\Sym_{X}^{n,\mathrm{sm}}\) obtained by desingularization. In particular, \(\Sym_{X}^{n,\mathrm{sm}}\) contains \(\Sym_{X}^{n,o}\) as a dense open subset. Moreover, if we can show that \(\Sym_{X\times \PP^{1}}^{n,\mathrm{sm}}\) satisfies HPBMrp (resp.\ WABMrp), the same argument as before will imply the HPBMrp (resp.\ WABMrp) property for \(\Sym_{X}^{n,\mathrm{sm}}\).
		
		Now suppose we are given a family of local points \(\{p_{v}\} \in \prod_{v\in\Omega_k}\Sym_{X\times \PP^{1}}^{n,\mathrm{sm}}(k_{v})\) that is orthogonal to \(\Br\!\bigl(\Sym_{X\times \PP^{1}}^{n,\mathrm{sm}}\bigr)/\Br(k)\), together with a finite set \(S\subset\Omega_{k}\). Since \(\Br(X)/\Br(k)\) is finite, we fix a finite set of generators \(\{b_{i}\}\). Enlarging \(S\) if necessary, we may assume that for every generator \(b_{i}\) and every local point \(x_{v}\in \Sym_{X\times \PP^{1}}^{n,\mathrm{sm}}(k_{v})\) with \(v\notin S\), the pairing \(\langle \lambda(b_{i}), x_{v}\rangle\) vanishes.
		
		We now approximate the family \(\{p_{v}\}\). Following \cite[Proposition 6.1]{CPS16}, we can replace \(\{p_{v}\}\) by a family \(\{q_{v}\}\) whose members all lie in the open subset \(\Sym_{X\times \PP^{1}}^{n,o}\) and still satisfy the orthogonality condition with respect to \(\Brnr\!\bigl(\Sym_{X\times \PP^{1}}^{n,o}\bigr)/\Br(k)\). The family \(\{q_{v}\}\) corresponds to a family of local effective separable \(0\)-cycles \(\{z_{v}\}\) on \(X\times \PP^{1}\). By Theorem~\ref{thm:main}, \(\{z_{v}\}\) is then orthogonal to \(\Br(X\times \PP^{1})\).
		
		Applying Lemma~\ref{moving lemma}, we obtain for each \(v\in S\) an effective separable \(0\)-cycle \(z_{v}'\) that is sufficiently close to \(z_{v}\) and whose image under the natural projection \(\pi: X\times \PP^{1}\rightarrow \PP^{1}\) is separable. Denote the resulting family by \(\{z_{v}^{1}\}\). Because each \(z_{v}^{1}\) is sufficiently close to the original \(z_{v}\) for all \(v\in S\), the new family remains orthogonal to \(\Br(X\times \PP^{1})\).
		
		Consider the push‑forwards \(\pi_{*}(z_{v}^{1})\) for \(v\in S\). They are effective separable \(0\)-cycles of degree \(n\) on \(\PP^{1}\). Let \(\Hil\subset \PP^1\) be the generalised Hilbertian subset defined by the finite cover $\mathbb{A}^1_L\to\mathbb{A}^1$. Applying Lemma~\ref{Hilbert_irred} to \(\Hil\)  and the family \(\{\pi_{*}(z_{v}^{1})\}_{v\in S}\), we obtain a closed point \(\theta\in\Hil\) that is, as a \(0\)-cycle, sufficiently close to \(\pi_{*}(z_{v}^{1})\) for every \(v\in S\).
		
		Write the local decomposition of \(\theta\) at a place \(v\) as \(\theta_{v}=\sum_{w\mid v,\, w\in\Omega_{k(\theta)}} P_{w}\) (using the local class field theory). Because \(\theta_{v}\) is sufficiently close to \(z_{v}^{1}\), we can write \(z_{v}^{1}\) in the form \(z_{v}^{1}=\sum_{w\mid v} Q_{w}\) with \(k(P_{w})=k(Q_{w})=k(\theta)_{w}\). By construction we may further express \(z_{v}^{1}=\sum_{w\mid v} M_{w}\) where each \(M_{w}\) lies in the fiber over \(Q_{w}\) and satisfies \(k(M_{w})=k(\theta)_{v}\). Applying the implicit function theorem, for each \(v\in S\) we can choose a point \(M_{w}^{0}\) in the fiber over \(P_{w}\) that is sufficiently close to \(M_{w}\). For places \(w\notin S\otimes k(\theta)\) we simply fix an arbitrary \(k(\theta)_{w}\)-point \(M_{w}^{0}\) of the fiber \(\pi^{-1}(\theta)=X_{k(\theta)}\). The family \(\{M_{w}^{0}\}\) then corresponds to a family of local points of \(\Sym_{X\times \PP^{1}}^{n,\mathrm{sm}}\) that is itself sufficiently close to the original family \(\{p_{v}\}\).
		
		Observe that because we chose each \(M_{w}^{0}\) sufficiently close to \(M_{w}\) and \(\Br(X)/\Br(k)\) is finite, we may additionally assume that \(\langle b_{i}, M_{w}\rangle = \langle b_{i}, M_{w}^{0}\rangle\) holds for every generator \(b_{i}\). Thus, as an adelic point of \(X_{k(\theta)}\), we have
		\[
		\sum_{w\in\Omega_{k(\theta)}} \langle b_{i}, M_{w}^{0}\rangle = 0
		\]
		for all \(b_{i}\in\Br(X)/\Br(k)\). Since $\theta \in \Hil$, the hypothesis of the theorem guarantees that the natural homomorphism \(\Br(X)/\Br(k)\to\Br(X_{k(\theta)})/\Br(k(\theta))\) is surjective. Therefore the adelic point \(\{M_{w}^{0}\}\) is orthogonal to \(\Br(X_{k(\theta)})/\Br(k(\theta))\).
		
		Finally, because \(X_{k(\theta)}\) satisfies HPBMrp (resp.\ WABMrp) by assumption, we can find a global point of \(X_{k(\theta)}\) (which, in the WABMrp case, can be taken arbitrarily close to the local points \(M_{w}^{0}\) for \(w\in\Omega_{k(\theta)}\)). This global point yields a global point of \(\Sym_{X\times \PP^{1}}^{n,\mathrm{sm}}\) (resp.\ a point that is sufficiently close to the original family \(\{p_{v}\}\)), thereby establishing the HPBMrp (resp.\ WABMrp) property for rational points on \(\Sym_{X\times \PP^{1}}^{n,\mathrm{sm}}\).
	\end{proof}
	
	\begin{remark}
		Let \(X\) be a smooth, projective, geometrically integral variety over a number field \(k\) such that \(\Pic(X_{\overline{k}})\) is torsion‑free. Assume furthermore that \(H^2(X,\mathcal{O}_X)=0\). Then  the Brauer group \(\Br(X)/\Br(k)\) is finite, combining the isomorphism given in Lemma~\ref{Lemma:Brauerregularextension}(c) and the Proposition \cite[Proposition~4.1]{HW16}, there exists a generalized Hilbertian subset \(\Hil\subset\PP^1\) such that for every residue field of a certain closed point of \(\Hil\), the natural restriction homomorphism
		\[
		\Br(X)/\Br(k) \longrightarrow \Br(X_K)/\Br(K)
		\]
		is surjective. Consequently, the hypotheses of Theorem~\ref{thm:generalization} reduce to the following simpler statement:
		
		\begin{quote}
			There exists a generalized Hilbertian subset \(\Hil\subset\PP^1\) such that for every field \(K\) corresponding to a point of \(\Hil\), the variety \(X_K\) satisfies HPBMrp (resp.\ WABMrp).
		\end{quote}
	\end{remark}

		\subsection*{Acknowledgements} The authors would like to thank Yang Cao for his useful advice.

	\bibliographystyle{amsalpha}
	\bibliography{references}
\end{document}